\documentclass[10pt,a4paper,reqno]{amsart}

\usepackage{amssymb,amsmath,mathrsfs}
\usepackage{enumitem}
\usepackage{amscd}
\usepackage{verbatim}
\numberwithin{equation}{section}

\theoremstyle{plain}

\newtheorem{theorem}[subsection]{Theorem}
\newtheorem{proposition}[subsection]{Proposition}
\newtheorem{lemma}[subsection]{Lemma}
\newtheorem{corollary}[subsection]{Corollary}

\theoremstyle{definition}

\newtheorem{definition}[subsection]{Definition}
\newtheorem{remark}[subsection]{Remark}

\newtheorem{example}[subsection]{Example}
\newtheorem{examples}[subsection]{Examples}

\newcommand\R{\mathbb{R}}
\newcommand\N{\mathbb{N}}
\renewcommand\P{\mathcal{P}}
\newcommand\p{\mathbb{P}}
\newcommand\F{\mathscr{F}}
\newcommand\E{\mathcal{E}}
\newcommand\G{\mathscr{G}}
\newcommand\X{\mathcal{X}}
\newcommand\D{\mathcal{D}}

\newcommand\B{\mathcal{B}}
\newcommand\Y{\mathcal{Y}}

\newcommand\eps{\varepsilon}

\begin{document}

\title{Measures and integrals in conditional set theory}

\author{Asgar Jamneshan}
\address{Department of Mathematics and Statistics, University of Konstanz}
\email{asgar.jamneshan@uni-konstanz.de}
\thanks{A.J.~and M.K.~gratefully acknowledge financial support from DFG project KU 2740/2-1.  The authors would like to thank an anonymous referee for a careful reading of the manuscript and helpful comments and suggestions.}

\author{Michael Kupper}
\address{Department of Mathematics and Statistics, University of Konstanz}
\email{kupper@uni-konstanz.de}

\author{Martin Streckfu\ss}
\address{Department of Mathematics, Humboldt University Berlin}
\email{streckfu@math.hu-berlin.de}

\begin{abstract}  
The aim of this article is to establish basic results in a conditional measure theory. 
The results are applied to prove that arbitrary kernels and conditional distributions are represented by measures in a conditional set theory. 
In particular, this extends the usual representation results for separable spaces.  
\end{abstract}


\maketitle

\section{Introduction}\label{s:introduction}

A random variable  $\xi$ is a measurable function from a probability space $(\Omega,\F,\p)$ into a measurable space $(E,\mathcal{E})$.  
Given a sub-$\sigma$-algebra $\G\subseteq \F$, the conditional distribution of $\xi$ given $\G$, i.e.~the quantity 
$\mathbb{E}[1_{\{\xi\in A\}}|\G](\omega)$ for $A\in \mathcal{E}$ and $\omega\in \Omega$, where $\mathbb{E}[\cdot|\G]$ denotes conditional expectation,  can be computed via a probability kernel $\kappa\colon \Omega\times \mathcal{E}\to [0,1]$ whenever $(E,\mathcal{E})$ is a standard Borel space.  
In particular, one has
\begin{equation*}\label{eq87}
\mathbb{E}[f(\xi)|\G](\omega)=\int f(x) \kappa(\omega,dx) \quad \text{almost surely,}
\end{equation*}
for every measurable function $f\colon E\to \mathbb{R}$ such that $f(\xi)$ is integrable.  
By a conditional measure theory, one can extend the previous representation results to random variables with values in an arbitrary measurable space.  
More precisely, we will show that a conditional distribution corresponds to a real-valued measure in conditional set theory and that $\mathbb{E}[f(\xi)|\G]$ can be computed with the help of a conditional Lebesgue integral. 
Conditional measure theory suggests a natural formalism to study kernels and conditional distributions in a purely measure-theoretic context without unnecessary topological assumptions such as separability. 

We will differ from the abstract setting in \cite{DJKK13}, where conditional set theory is developed relative to an arbitrary complete Boolean algebra.  
Motivated by the aforementioned applications, we construct a conditional measure theory relative to the complete Boolean algebra obtained from a probability space by quotienting out null sets.
In accordance with this aim,  we restrict attention to a class of conditional sets which are associated to spaces of random variables stable under countable concatenations. 
By a straightforward generalization, one can extend all results in this article to the abstract setting in \cite{DJKK13}.

\subsection*{Our contribution}  
We establish basic results in conditional measure theory, e.g.~a $\pi$-$\lambda$-theorem, a Carath\'eodory extension theorem, construction of a Lebesgue measure and a Lebesgue integral, see Section  \ref{s:measure}.  
These results are applied to connect kernels and conditional distributions with measures in a conditional theory in Section \ref{s:kernel}. 
A conditional version of some standard theorems in measure theory is proved in Section \ref{s:theorems}.

\subsection*{Related literature}
A conditional set is an abstract set-like structure, see \cite{DJKK13} for a thorough introduction. 
Conditional set theory is closely related to Boolean-valued models and topos theory, see \cite{jamneshan2014topos}. 
Basic results for Riemann integration in a Boolean-valued model are studied in \cite[Chapter 2]{takeuti2015two}.  
See \cite{Cheridito2012,DKKM13,kupper03,frittelli2014conditional,jamneshan2017compact,OZ2017stabil} for further results in conditional analysis and conditional set theory. 
For applications of  conditional analysis, we refer to \cite{BH14,martinsamuel13,ch11,cs12,DJ13,drapeau2017fenchel,kupper11,frittelli14,HR1987conditioning,JKZ2017control}. See \cite{kuppermaccheroni, guo10} for other related results.

\section{Preliminaries}\label{s:notation}

Throughout, we fix a complete probability space $(\Omega,\F,\p)$, and identify $A,B\in \F$ whenever $\p(A\Delta B)=0$, where $\Delta$ denotes symmetric difference. 
In particular, $A\subseteq B$ is understood in the almost sure (a.s.) sense. 
Let $\mathscr{F}_+$ denote the set of all $A\in\mathscr{F}$ with $\p(A)>0$. 
For a function $x$ on $\Omega$ with values in an arbitrary set $E$ and $A\in \F$, we denote by $x|A\colon A\to E$ the restriction of $x$ to $A$. 
If $x,y\colon \Omega \to E$ are functions and $A,B\in \F$, then we always identify $x|A$ with $y|B$ whenever $\p(A\Delta B)=0$ and $x(\omega)=y(\omega)$ for a.a. $\omega\in A \cap B$. 
A partition is a countable family $(A_k)$ of measurable subsets of $\Omega$ which is pairwise disjoint and such that $\cup_k A_k =\Omega$. 
Given a sequence of functions $x_k\colon \Omega\to E$ and a partition $(A_k)$, we denote by $\sum x_k|A_k$ the unique function $x\colon \Omega \to E$ with the property that $x|A_k=x_k|A_k$ for all $k$. 
For a set $X$ of functions $x\colon\Omega\to E$ and $A\in \F$, let $X|A:=\{x|A\colon x\in X\}$, and define $(X|A)|B:=X|A\cap B$ for all $B\in\F$. 
Notice that $X|\emptyset=\{\ast\}$ is a singleton and trivially $X|\Omega=X$. 

Denote by $\bar L^0=\bar L^0(\Omega,\F,\mathbb{P})$ the set of all measurable functions $x\colon \Omega\to [-\infty,\infty]$, and by $L^0$ its subset of all real-valued functions. 
On $\bar L^0$ consider the complete order $x\leq y$ a.s. 
Recall that on $L^0$ this order is Dedekind complete. 
Let $\bar L^0_+:=\{x\in \bar L^0\colon x\geq 0\}$, $L^0_+:=\{x\in  L^0\colon x\geq 0\}$ and $L^0_{++}:=\{x\in L^0\colon x>0\}$.  
For an arbitrary subset $\G\subseteq \F$, its supremum w.r.t.~almost sure inclusion is the measurable set $A\in \F$ (unique mod a.s. identification) such that  $1_A=\sup \{1_{A^\prime}\colon A^\prime\in \G\}$.  
Similarly, we define the infimum of $\G$, and denote them by $\sup \{A^\prime\colon A^\prime\in \G\}$ and   $\inf \{A^\prime\colon A^\prime\in \G\}$, respectively.  
Throughout, all order relations between extended real-valued random variables are understood in the a.s.~sense.  
\begin{definition}
A set $X$ of functions on $\Omega$ with values in a fixed image space $E$  is said to be \emph{stable under countable concatenations}, or \emph{stable}  for short, if $X\neq\emptyset$ and $\sum x_k|{A_k} \in X$ for every partition $(A_k)$ and each sequence $(x_k)$ in $X$.  
\end{definition}
The spaces $\bar L^0$, $L^0$, $\bar L^0_+$, $L^0_+$ and $L^0_{++}$ are stable sets of functions on $\Omega$. 
We provide more examples which are important in the sequel. 
\begin{examples}\label{exp:sets}
\begin{itemize}[fullwidth]
\item[1)] Let $E$ be a  nonempty  set,  $(x_k)$ a sequence in $E$ and $(A_k)$ a partition. Let $\sum x_k|A_k$ denote the function on $\Omega$ with value $x_k$ on $A_k$ for each $k$ and call it a \emph{step function} with values in $E$. The set of all step functions with values in $E$ is a stable set of functions on $\Omega$ and is  denoted by $L^0_s(E)$. 
In particular, $L^0_s(\N)$ denotes the set of all step functions with values in the natural numbers. 
\item[2)] Let $(E,\mathscr{E})$ be an arbitrary measurable space. The set of all measurable functions  on $\Omega$ with values in $E$ is stable and is denoted by $L^0_m(E)$. 
\item[3)]  If $E$ is a topological space and $\mathscr{E}$ its Borel $\sigma$-algebra. 
The set of all strongly measurable\footnote{Strongly measurable means Borel measurable and essentially separably-valued.} functions on $\Omega$ with values in $E$ is stable and is denoted by $L^0(E)$. 
\item[4)] Every nonempty order-bounded  subset of $L^p$ for $0\leq p\leq \infty$ is stable. 
\item[5)] Given a  nonempty  subset $V$ of a stable set of functions on $\Omega$, let 
\[
\text{st}(V):=\left\{\sum x_k|A_k\colon (A_k)\text{ partition, } (x_k) \text{ sequence in } V\right\}. 
\]
We call $\text{st}(V)$ the \emph{stable hull} of $V$. 
\item[6)] Let $X$ be a stable set of functions on $\Omega$, $(V_k)$ a sequence of stable subsets of $X$ and $(A_k)$ a partition. 
Then 
\begin{equation}\label{eq:setconcatenation}
\sum V_k|A_k:=\left\{\sum x_k|A_k\colon x_k\in V_k \text{ for each } k\right\}
\end{equation}
is a stable subset of $X$. 
\item[7)] Let $X_i$ be a stable set of functions on $\Omega$ with image spaces $E_i$ for each $i$ in an index. Their \emph{Cartesian product}
\[
\prod X_i:=\left\{x\colon \Omega\to \prod E_i \colon  x(\omega)=(x_i(\omega)), x_i\in X_i \text{ for each }i\right\}
\]
\end{itemize}
 is a stable set of functions on $\Omega$. 
\end{examples}
\begin{remark}\label{r:sep}
In general, $L^0(E)$ is strictly smaller than $L^0_m(E)$ for a topological space $E$.  
However, if $(E,\mathscr{E})$ is a standard Borel space, then $L^0(E)$ and $L^0_m(E)$ coincide.  
\end{remark}
We have a close connection between stable sets of functions on $(\Omega,\F,\p)$ and conditional sets of the associated measure algebra which will allows us to apply the machinery of conditional set theory \cite{DJKK13}. 
We start by defining an analog of the conditional power set. 
\begin{definition}
Let $X$ be a stable set of functions on $\Omega$. 
Let $\P(X)$ denote the collection of all \emph{conditional subsets} $V|A$, where $V$ is a stable subset of $X$ and $A\in \F$. 
A countable concatenation in $\mathcal{P}(X)$ is defined by 
\begin{equation}\label{eq:stabilitypowerset}
\sum (V_k|B_k)|A_k:= (\sum V_k|A_k)|\cup_k (A_k \cap B_k) \in \mathcal{P}(X),  
\end{equation}
for a sequence $(V_k|B_k)$ and a partition $(A_k)$, where $\sum V_k|A_k$ is defined in \eqref{eq:setconcatenation}. 
A subset of $\P(X)$ is called a \emph{stable collection} if it is nonempty and closed under countable concatenations.  
\end{definition}
In the present context, the conditional inclusion relation \cite[Definition 2.9]{DJKK13} reads as follows. 
\begin{definition}\label{d:inclusion}
The \emph{conditional inclusion relation} on $\mathcal{P}(X)$ is the binary relation 
\[
V|A \sqsubseteq W|B \text{ if and only if } A\subseteq B \text{ and } V|A \subseteq W|A. 
\]
\end{definition}
An important result in conditional set theory which will be helpful for the construction of a conditional measure theory is the following theorem which is proved in \cite[Theorem 2.9]{DJKK13}. 
\begin{theorem}\label{t:mainthm}
The ordered set $(\mathcal{P}(X),\sqsubseteq)$ is a complete complemented distributive lattice. 
\end{theorem}
\begin{proof}  
For the sake of completeness, we provide the main constructions of the proof in the present setting. 
Notice that $X$ is the greatest and $X|\emptyset=\{\ast\}$ the least element in $\mathcal{P}(X)$. 
Let $(V_i|A_i)$ be a  nonempty  family in $\mathcal{P}(X)$. 
We construct its supremum and infimum w.r.t.~$\sqsubseteq$. 
Put $A:=\sup_i A_i$ and fix some $x_0\in X$. 
Define 
\[
V:=\left\{\sum x_k|B_k + x_0|A^c \colon \text{ for all }k \text{ there is } i \text{ s.t. }B_k\subseteq A_{i} \text{ and } x_k\in V_{i}  \text{ and }(B_k) \text{ is a partition of } A\right\}. 
\]
The conditional subset 
\begin{equation}\label{eq:union}
\sqcup_i V_i|A_i:=V|A
\end{equation}
is the supremum of $(V_i|A_i)$. 
As for its infimum, let 
\[
\mathscr{G}=\{A\in\mathscr{F}\colon A\subseteq \inf_i A_i \text{ and } \cap_i V_i|A\neq \emptyset\}. 
\]
Put $A_\ast=\sup \mathscr{G}$. We show that $A_\ast$ is attained. Let $(A_k)$ be a sequence in $\mathscr{G}$ such that $A_\ast=\sup_k A_k$. 
Define $B_1=A_1$ and $B_k=A_k\cap (B_1\cup \ldots \cup B_{k-1})^c$ for $k\geq 2$.  
Then $(B_k)$ is a partition of $A_\ast$, and since $\cap_i V_i|B_k\neq \emptyset$ for all $k$, it follows from the stability of the $V_i$'s that $\cap_i V_i|A_\ast \neq \emptyset$. 
Fix $x_0\in X$, and modify each $V_i$ by $W_i=V_i|A_\ast+\{x_0\}|A_\ast^c$. Then $\emptyset\neq\cap_i W_i=:W$ is stable, and 
\begin{equation}\label{eq:intersection}
\sqcap_i V_i|A_i:=W|A_\ast
\end{equation}
is the infimum of $(V_i|A_i)$. 

We provide the construction of the complement. 
For $V|A\in \mathcal{P}(X)$, define 
\begin{equation}\label{eq:complement}
(V|A)^\sqsubset:=\sqcup\{W|B\in\mathcal{P}(X)\colon W|B\sqcap V|A=\{\ast\}\}. 
\end{equation}
By completeness of the ordered set $(\P(X),\sqsubseteq)$, $(V|A)^\sqsubset$ is well-defined. 
\end{proof}
The symbols $\sqcup, \sqcap$ and ${}^\sqsubset$, defined in \eqref{eq:union}, \eqref{eq:intersection} and \eqref{eq:complement}, are called \emph{conditional union, conditional intersection} and \emph{conditional complement}, respectively. 
We give another description of the conditional complement: 
\begin{equation}\label{eq:repcomplement}
(V|A)^\sqsubset=Z|B_\ast + X|A^c, 
\end{equation}
where 
\[
B_\ast=\sup\{A^\prime\in \mathcal{F}\colon A^\prime\subseteq A, V|A^\prime\neq X|A^\prime\} 
\]
and 
\begin{equation*}
Z=
\begin{cases}
\{x\in X\colon x|A^\prime\not\in V|A^\prime \text{ for all } A^\prime\in\F_+ \text{ with }A^\prime\subseteq B_\ast\}, \text{if }B_\ast\in \F_+, \\
X, \text{ else}. 
\end{cases}
\end{equation*}
By an exhaustion argument (similarly to the construction of the conditional intersection in the previous proof), $B_\ast$ is attained. 
It can directly be verified that $Z$ is stable. 
To see the equality in \eqref{eq:repcomplement}, notice that $(Z|B_\ast + X|A^c)\sqcap V|A=\{\ast\}$, and therefore $Z|B_\ast + X|A^c\sqsubseteq (V|A)^\sqsubset$. 
By way of contradiction, suppose that $(Z|B_\ast + X|A^c) \sqcap (V|A)^\sqsubset \neq \{\ast\}$. We may assume that $B_\ast\in \F_+$. 
Since $(V|A)^\sqsubset|A^c=X|A^c$, there must exist $C\in \F_+$ with $C\subseteq B_\ast$ and $x\in X$ such that $x|C\in (V|A)^\sqsubset$ and $x|C^\prime\not\in V|C^\prime$ for all $C^\prime\subseteq C$ with $C^\prime\in \F_+$. However, this contradicts the maximality of $B_\ast$. 

As a consequence of elementary results in the theory of Boolean algebras, see e.g.~\cite[p.~14]{monk1989handbook}, one has: 
\begin{corollary}
    For a stable set $X$, the structure $(\mathcal{P}(X),\sqcup,\sqcap,{}^\sqsubset, \{\ast\},X)$ is a complete Boolean algebra.
\end{corollary}
The importance of the previous result lies in the validity of all Boolean laws in $\P(X)$ which are known from the classical power set algebra. The Boolean laws are fundamental in topology and measure theory. 
The application of the Boolean laws is usually referred to as Boolean arithmetic, see e.g.~\cite{monk1989handbook}. 
Compared with the usual power set, the conditional power set has additionally the property that its elements can be concatenated which 
 brings some new properties which are specific to the conditional power set and are summarized in the following proposition.  
\begin{proposition}\label{p:elementaryprop}
Let $X$ be a stable set, $V|A\in\mathcal{P}(X)$, $(V_i|A_i)_{i\in I}$, $(V_k|A_k)_{k\in \N}$ and $(V_{k,j}|A_{k,j})_{k\in \N, j\in J}$ be families in $\mathcal{P}(X)$, where $I$ and $J$ are arbitrary  nonempty  index sets. Let $(D_k)$ be a partition of $\Omega$. Then the following are true. 
\begin{itemize}
\item[(S1)] $\sqcap_i (V_i|A_i)= (\sqcap_i V_i)|\inf_i A_i$ and $\sqcup_i (V_i|A_i) = (\sqcup_i V_i)|\sup_i A_i$. 
\item[(S2)] $(\sum (V_{k}|A_k)|D_k)\sqcap V|A=\sum (V_{k}|A_k\sqcap V|A)|D_k$ and $(\sum (V_{k}|A_k)|D_k)\sqcup V|A=\sum (V_{k}|A_k\sqcup V|A)|D_k$.  
\item[(S3)] $\sum (\sqcap_j V_{k,j}|A_{k,j})|D_k=\sqcap_j \sum (V_{k,j}|A_{k,j})|D_k$ and $\sum (\sqcup_j V_{k,j}|A_{k,j})|D_k=\sqcup_j \sum (V_{k,j}|A_{k,j})|D_k$.
\item[(S4)] $(\sum (V_{k}|A_{k})| D_k)^\sqsubset=\sum (V_{k}|A_{k})^\sqsubset|D_k$.  
\end{itemize}
\end{proposition}
\begin{proof}
We prove (S1) and (S4); the remaining two claims can be shown similarly by using the definitions of the conditional set operations.  
\begin{itemize}
\item[(S1)] We show $\sqcap_i (V_i|A_i)= (\sqcap_i V_i)|\inf_i A_i$. Suppose $\sqcap_i (V_i|A_i)=W|A_\ast$ and $\sqcap_i V_i=W^\prime|B_\ast$ according to \eqref{eq:intersection}, and put $C_\ast=B_\ast\cap \inf_i A_i$. 
If $x|A_\ast\in W|A_\ast$, then $x|A_\ast\in W_i|A_\ast$ for all $i$, and thus $x|A_\ast\in W^\prime|A_\ast\sqsubseteq W^\prime|C_\ast$ since $A_\ast\subseteq C_\ast$. If now $C_\ast\setminus A_\ast\in\F_+$, then this contradicts the maximality of $A_\ast$, which implies that $W^\prime|C_\ast \sqsubseteq W|A_\ast$. 

We show $\sqcup_i (V_i|A_i)= (\sqcup_i V_i)|A$, where $A=\sup_i A_i$. Suppose $\sqcup_i (V_i|A_i)=W|A$ and $\sqcup V_i=W^\prime$ according to $\eqref{eq:union}$. 
By stability, it follows that $W|A\subseteq W^\prime|A$. 
As for the converse, let $(A_{i_k})$ be a countable subfamily of $(A_i)$ such that $\sup_k A_{i_k}=A$ and $\sum x_h|B_h\in W^\prime$ for some partition $(B_h)$ of $\Omega$ and $x_h\in X_{i}$ for some $i$ and all $h$. 
Then $(B_h\cap A_{i_k})_{h,k}$ is a partition of $A$, and extending the family $(x_h)$ to a family $(x_{h,k})$ by $x_{h,k}:=x_h$ for all $h,k$ we obtain $(\sum x_h|B_h)|A=\sum x_{h,k}|B_h\in W|A$, which shows that $W^\prime|A\subseteq W|A$. 
\item[(S4)] According to \eqref{eq:repcomplement}, let $(V_k|A_k)^\sqsubset=U_k|B_{k,\ast}+ X|A_k^c$ for each $k$, and 
\[
(\sum (V_k|A_k)|D_k)^\sqsubset=((\sum V_k|D_k)|A)^\sqsubset=U|B_\ast + X|A^c, 
\]
where $A=\cup_k (A_k\cap D_k)$ and the first equality follows from \eqref{eq:stabilitypowerset}. 
On the one hand, we have 
\[
(\sum (U_k|B_{k,\ast}+ X|A_k^c)|D_k)|\cup_k (A^c_k\cap D_k)=X|A^c. 
\]
On the other hand, $U|B_\ast\cap (A_k\cap D_k)=U_k|B_{k,\ast}\cap D_k$ for all $k$, which implies 
\[
U|B_\ast=(\sum U_k|D_k)|\cup_k (B_{k,\ast}\cap D_k). 
\]
\end{itemize}
\end{proof}
\begin{remark}\label{r:exhaustion}
In many concepts and proofs in conditional set theory, it is necessary to construct the largest set $A\in \F$ on which a property is satisfied. 
For example see the definition of the conditional intersection \eqref{eq:intersection} and the representation of the conditional complement \eqref{eq:repcomplement}. 
This can be achieved by an \emph{exhaustion argument} which requires that the considered property is stable under countable concatenations which follows from the stability of sets and relations. 
Without further notice, if we write ``by an exhaustion argument'', then we mean that an exhaustion can is applied. 
\end{remark}
Finally, we recall the definition of a stable function, cf.~\cite[Definition 2.17]{DJKK13}. 
\begin{definition}
Let $X$ and $Y$ be stable sets. A function $f\colon X\to Y$ is said to be a \emph{stable function} if $f(\sum x_k|A_k)=\sum f(x_k)|A_k$ for all partitions $(A_k)$ and every sequence $(x_k)$ in $X$. 
Let $W|A\sqsubseteq Y$ and put
\[
C_\ast:=\sup\{A^\prime\in\mathscr{F}\colon A^\prime\subseteq A, \text{ there is }x\in X \text{ such that } f(x)|A^\prime\in W|A^\prime\}. 
\]
By an exhaustion argument,  $C_\ast$ is attained. 
Let 
\[
V:=\{x\in X\colon f(x)|C_\ast\in W|C_\ast\}. 
\]
Define the \emph{conditional pre-image} of $W|A$ as 
\begin{equation}\label{eq:preimage}
f^{-1}(W|A):=V|C_\ast. 
\end{equation}
For a sequence of stable functions $f_k\colon X\to Y$ and a partition $(A_k)$, define their concatenation by 
\begin{equation}\label{eq:concatenationfct}
x\mapsto (\sum f_k|A_k)(x):=\sum f_k(x)|A_k, 
\end{equation}
which is a stable function from $X$ to $Y$.  
\end{definition}

\section{A conditional version of Carath\'eodory's extension theorem and the Lebesgue integral}\label{s:measure}

The aim of this section is to develop basic results of measure theory in a conditional context.  
Our focus lies on Carath\'eodory's extension and uniqueness theorems and the construction of a Lebesgue integral. 
\begin{definition}
Let $X$ be a stable set of functions on $\Omega$. 
A stable collection $\X$ on $X$ is called a
\begin{itemize}
 \item \emph{stable ring} whenever $V|A\sqcap (W|B)^\sqsubset,V|A\sqcup W|B\in \mathcal{X}$ for all $V|A,W|B\in \mathcal{X}$;
 \item \emph{stable Dynkin system} whenever $X\in \mathcal{X}$, $(V|A)^\sqsubset\in \mathcal{X}$ for all $V|A\in \mathcal{X}$, and $\sqcup_k (V_k|A_k) \in \mathcal{X}$ for all sequences $(V_k|A_k)$ of pairwise disjoint\footnote{$V|A$ and $W|B$ are said to be \emph{disjoint} if $V|A\sqcap W|B=\{\ast\}$.}  elements in $\mathcal{X}$;
 \item \emph{stable $\sigma$-algebra} whenever $X\in \mathcal{X}$, $(V|A)^\sqsubset\in \mathcal{X}$ for all $V|A\in \mathcal{X}$, and $\sqcup_k (V_k|A_k) \in \mathcal{X}$ for all sequences $(V_k|A_k)$ in $\mathcal{X}$. 
\end{itemize}
The pair $(X,\mathcal{X})$, where $\mathcal{X}$ is a stable $\sigma$-algebra, is called a \emph{stable measurable space}. 
If $(Y,\mathcal{Y})$ is another stable measurable space, then a stable function $f\colon X\to Y$ is called \emph{stably measurable} if $f^{-1}(V|A)\in\mathcal{X}$ for all $V|A\in \mathcal{Y}$. 
\end{definition}
\begin{remark}
Since the intersection of a  nonempty  family of stable Dynkin systems (stable $\sigma$-algebras) is a stable Dynkin system (stable $\sigma$-algebra), we can define the stable Dynkin system (stable $\sigma$-algebra)  \emph{generated} by a  collection $\mathcal{E}$ of conditional sets, henceforth denoted by $\D(\mathcal{E})$ ($\Sigma(\mathcal{E})$). 
\end{remark}
\begin{examples}\label{exp:algebra}
\begin{itemize}[fullwidth]
\item[1)] Let $X$ be a stable set of functions on $\Omega$. Then $\{X|A\colon A\in \F\}$ and $\P(X)$ are stable $\sigma$-algebras, called the \emph{trivial} and the \emph{discrete} stable $\sigma$-algebra on $X$. 
\item[2)] Let $L^0(E)$ be the space of all strongly measurable functions with values in a Banach space $E$. The norm $\|\cdot\|$ of $E$ extends to an $L^0$-valued norm on $L^0(E)$ by defining $\|x\|(\omega):=\|x(\omega)\|$ a.s. Let $\mathcal{E}$ be the collection of all \emph{stable open balls} $B_r(x):=\{y\in L^0(E)\colon \|x-y\|< r\}$, $x\in L^0(E)$ and $r\in L^0_{++}$. By straightforward inspection, $\mathcal{E}$ is a stable collection in $\mathcal{P}(L^0(E))$ which is a base for a stable topology on $L^0(E)$, see \cite[Section 3]{DJKK13}. We call the stable $\sigma$-algebra $\Sigma(\mathcal{E})$ the \emph{stable Borel $\sigma$-algebra} on $L^0(E)$, and denote it by $\mathcal{B}(L^0(E))$. 
\end{itemize}
\end{examples}
The following proposition is an immediate consequence of Boolean arithmetic. 
\begin{proposition}\label{cor1}
A stable Dynkin system is a stable $\sigma$-algebra if and only if it is closed under finite conditional intersections. 
\end{proposition}
The previous results can be used to establish the following Dynkin's $\pi$-$\lambda$ type result.  
\begin{theorem}\label{pilambda}
Let $\mathcal{E}$ be a stable collection which is closed under finite conditional intersections. 
Then $\Sigma(\mathcal{E})=\D(\mathcal{E})$.  
\end{theorem}
\begin{proof}
By Corollary \ref{cor1}, it is enough to prove that $\D(\mathcal{E})$ is closed under finite conditional intersections. 
For $V|A\in \D(\mathcal{E})$, let $\mathcal{F}_{V|A}:=\{W|B \in \P(X) \colon V|A \sqcap W|B \in \D(\mathcal{E})\}$. 
By (S3) and the stability of $\D(\mathcal{E})$, it follows that $\mathcal{F}_{V|A}$ is a stable collection. 
By  Boolean arithmetic,  it follows that $\mathcal{F}_{V|A}$ is a stable Dynkin system. 
For every $V|A \in\mathcal{E}$, one has $\mathcal{E}\subseteq \mathcal{F}_{V|A}$ from the assumption, and therefore $\D(\mathcal{E})\subseteq \mathcal{F}_{V|A} $. 
If now $V|A\in \mathcal{E}$ and $W|B\in \D(\mathcal{E})$, then $W|B\in \mathcal{F}_{V|A}$, and thus $V|A\in \mathcal{F}_{W|B}$. 
It follows that $\mathcal{E}\subseteq \mathcal{F}_{W|B}$ and $\D(\mathcal{E})\subseteq \mathcal{F}_{W|B}$ which shows the claim. 
\end{proof}
\begin{remark}
Let $(X,\X)$ and $(Y,\Y)$ be stable measurable spaces, $\mathcal{E}$ a collection in $\P(Y)$ generating $\Y$ and $f\colon X\to Y$ a stable function. 
Then $f$ is  stably measurable  if and only if $f^{-1}(V|A)\in \X$ for all $V|A\in \mathcal{E}$. 
Indeed, by stability of $f$, $\P(X)$ and $\P(Y)$, the collection 
\[
\mathcal{Z}=\{V|A\in \P(Y)\colon f^{-1}(V|A)\in \X\}
\]
is stable. 
One has $Y\in \mathcal{Z}$, and $\mathcal{Z}$ is closed under conditional complementation and countable conditional unions.  
Thus $\Y\subseteq \mathcal{Z}$ if and only if $\mathcal{E}\subseteq \mathcal{Z}$. 
\end{remark}
If $\X$ is a stable collection in $\P(X)$, then a function $\mu \colon \X\to \bar L^0_+$ is said to be a \emph{stable set function} if 
\[
\mu(\sum (V_k|B_k)|A_k)=\sum \mu(V_k|B_k)|A_k
\]
for all sequences $(V_k|B_k)$ in $\X$ and every partition $(A_k)$, where the concatenation on the l.h.s~is defined as in \eqref{eq:stabilitypowerset}. 
\begin{definition}\label{d:measure}
Let $X$ be a stable set  of functions on $\Omega$ and $\mathcal{X}$ a stable ring on $X$. 
A function $\mu:\mathcal{X}\to \bar L^0_+$ is called a \emph{stable pre-measure} whenever $\mu$ is a stable function and 
\begin{itemize}
\item[(M1)] $\mu(V|A)=\mu(V|A)|A + 0|A^c$ for all $V|A\in \mathcal{X}$; 
\item[(M2)] $\mu(\sqcup_k (V_k|A_k))=\sum_{k \geq 1}\mu(V_k|A_k)$ for all sequences $(V_k|A_k)$ of  pairwise disjoint elements in $\P(X)$ with $\sqcup_k (V_k|A_k)\in \mathcal{X}$.\footnote{The infinite series in $\bar L^0_+$ is understood pointwise a.s.}
\end{itemize} 
If the domain of a stable pre-measure $\mu$ is a stable $\sigma$-algebra, then $\mu$ is called a \emph{stable measure} and  the triple $(X,\mathcal{X},\mu)$ a \emph{stable measure space}. 
A stable measure $\mu$ is said to be 
\begin{itemize}
\item \emph{finite} if $\mu(X)<\infty$, and a \emph{stable probability measure} if $\mu(X)=1$; 
\item \emph{$\sigma$-finite} if there exists an increasing sequence $(V_k|A_k)$ in $\mathcal{X}$ with $\sqcup_k (V_k|A_k)=X$ such that $\mu(V_k|A_k)<\infty$ for all $k$.   
\end{itemize}
\end{definition}
One can see property (M1) as the local analogue of $\mu(\emptyset)=0$ for classical measures. 
We give one simple example of a stable measure and will provide further examples in Section \ref{s:kernel}. 
\begin{example}
 Let $(X,\X)$ be a stable measurable space and $x\in X$. 
 The \emph{stable Dirac measure} centered at $x$ is defined by
 \[
  \delta_x(V|A):= 1|B+0|B^c, 
 \]
where $B=\sup \{B^\prime\in \F\colon B^\prime\subseteq A,  x|B^\prime\in V|B^\prime\}$ which is attained by an exhaustion argument. 
\end{example}
Similarly to the classical case, we have the following elementary properties of stable pre-measures. 
\begin{proposition}
Let $X$ be a stable set of functions on $\Omega$, $\X$ a stable ring and $\mu\colon \X\to \bar L^0_+$ a stable function satisfying property (M1) and finite additivity (i.e.~(M2) for finite sequences). 
Then the following are true.  
\begin{itemize}
\item[(M3)] $\mu(V|A\sqcup W|B)+\mu(V|A\sqcap W|B)=\mu(V|A)+\mu(W|B)$ for all $V|A,W|B\in \mathcal{X}$; 
\item[(M4)] $\mu(V|A)\leq \mu(W|B)$ whenever $V|A\sqsubseteq W|B$; 
\item[(M5)] $\mu(W|B\sqcap (V|A)^\sqsubset)=\mu(W|B)-\mu(V|A)$ whenever $V|A\sqsubseteq W|B$ and $\mu(V|A)<\infty$; 
\item[(M6)] if $\mu$ is a stable pre-measure, then $\mu(\sqcup_k V_k|A_k)\leq \sum_k\mu(V_k|A_k)$ for all sequences $(V_k|A_k)$ in $\mathcal{X}$ with $\sqcup_k V_k|A_k\in \mathcal{X}$; 
\item[(M7)] $\mu$ is a stable pre-measure if and only if $\mu(V_k|A_k)\uparrow\mu(\sqcup_k V_k|A_k)$ a.s.~for all increasing sequences $(V_k|A_k)$ in $\mathcal{X}$ with $\sqcup_k V_k|A_k \in \mathcal{X}$; 
\item[(M8)] if $\mu(V|A)<\infty$ for all $V|A\in \mathcal{X}$, then $\mu$ is a stable pre-measure if and only if $\mu(V_k|A_k)\downarrow \mu(\sqcap_k V_k|A_k)$ a.s.~for all decreasing sequences $(V_k|A_k)$ in $\mathcal{X}$ with  $\sqcap_k V_k|A_k \in \mathcal{X}$.  
\end{itemize}
\end{proposition}
\begin{proof}
Since the verification of the statements is similar to the proof of the respective classical statements, we only provide the main arguments. 
\begin{itemize}
\item[(M3)] We have 
\begin{align*}
\mu(V|A\sqcup W|B)&=\mu(V|A)+\mu(W|B\sqcap (V|A)^\sqsubset),\\
\mu(W|B)&=\mu(V|A\sqcap W|B)+\mu(W|B\sqcap (V|A)^\sqsubset). 
\end{align*}
Let $C=\{\mu(W|B\sqcap (V|A)^\sqsubset)<\infty\}$. On $C$ we get (M3) by adding the previous two identities and subtracting $\mu(W|B\sqcap (V|A)^\sqsubset)$. 
On $C^c$ we have $\mu(V|A\sqcup W|B)=\mu(W|B)=\infty$ which also implies (M3). 
\item[(M4)] and (M5) follow from finite additivity. 
\item[(M6)] can be shown by finite additivity, (M4) and Boolean arithmetic. 
\item[(M7)] follows from finite additivity and Boolean arithmetic. 
\item[(M8)] follows from (M7) on applying (M5). 
\end{itemize}
\end{proof}
\begin{definition}
Let $X$ be a stable set of functions on $\Omega$. 
A stable function $\mu^\ast\colon\P(X)\to \bar L^0_+$ is called a \emph{stable outer measure} whenever 
\begin{itemize}
\item[(A1)] $\mu^\ast((V|A)|B)=\mu^\ast(V|A)|B+0|B^c$ for all $V|A\in \P(X)$ and $B\in\F$; 
\item[(A2)] $\mu^\ast(V|A)\leq \mu^\ast(W|B)$ whenever $V|A\sqsubseteq W|B$; 
\item[(A3)] $\mu^\ast(\sqcup_k (V_k|A_k))\leq\sum_k\mu^\ast(V_k|A_k)$ for all sequences $(V_k|A_k)$ in $\P(X)$.  
\end{itemize} 
An element $V|A\in \P(X)$ is said to be \emph{stably  $\mu^\ast$-measurable} whenever  
\begin{equation}\label{eq:measurable}
\mu^\ast(W|B\sqcap V|A) + \mu^\ast(W|B\sqcap (V|A)^\sqsubset)= \mu^\ast(W|B)\quad \text{ for all } W|B\in\P(X). 
 \end{equation} 
\end{definition}
A stable measure can be obtained from a stable outer measure as follows. 
\begin{proposition}\label{p:outermeas}
Let $X$ be a stable set of functions on $\Omega$ and $\mu^\ast\colon \P(X)\to \bar L^0_+$ a stable outer measure. 
Then $\mu^\ast$ is a stable measure on the stable collection $\X(\mu^\ast)$ of all stably $\mu^\ast$-measurable sets.  
\end{proposition}
 \begin{proof}
 From the stability of $\mu^\ast$ and (S2), we obtain that $\mathcal{X}(\mu^\ast)$ is a stable collection. 
Clearly, $X\in \X(\mu^\ast)$, and by symmetry, $\X(\mu^\ast)$ is closed under conditional complementation. 
Let $V|A, W|B\in \X(\mu^\ast)$. 
 Then 
 \begin{equation}\label{eq0}
  \mu^\ast(Z|C)=\mu^\ast(Z|C\sqcap W|B )+\mu^\ast(Z|C\sqcap (W|B)^\sqsubset) \quad \text{ for all }Z|C\in \P(X).  
 \end{equation}
 Replacing $Z|C$ by $Z|C\sqcap  V|A$ and $Z|C\sqcap  (V|A)^\sqsubset$ in \eqref{eq0}, respectively, and taking the sum of the two resulting equations, we obtain 
 \begin{align}\label{eq00}
  \notag \mu^\ast(Z|C)&=\mu^\ast(Z|C\sqcap V|A\sqcap W|B )+\mu^\ast(Z|C\sqcap  V|A\sqcap (W|B)^\sqsubset)\\
  &+\mu^\ast(Z|C\sqcap (V|A)^\sqsubset\sqcap W|B )+\mu^\ast(Z|C\sqcap (V|A)^\sqsubset\sqcap (W|B)^\sqsubset). 
 \end{align}
 Replace $Z|C$ by $Z|C\sqcap( V|A\sqcup W|B )$ in \eqref{eq00}, and get from Boolean arithmetic
 \begin{equation}\label{eq01}
 \mu^\ast(Z|C\sqcap( V|A\sqcup W|B ))=\mu^\ast(Z|C\sqcap V|A\sqcap W|B )+\mu^\ast(Z|C\sqcap V|A\sqcap W|B ^\sqsubset)+\mu^\ast(Z|C\sqcap V|A^\sqsubset\sqcap W|B ).
 \end{equation}
 Plugging in \eqref{eq01} in \eqref{eq00}, one has $V|A\sqcup W|B \in\mathcal{X}(\mu^\ast)$. 
 By de Morgan's law, $\X(\mu^\ast)$ is also closed under finite intersections. 
 Let $( U_k|D_k )$ be a sequence of pairwise disjoint sets in $\mathcal{X}(\mu^\ast)$ and  $U|D=\sqcup_k U_k|D_k $.  
 Choosing $ V|A=U_1|D_1$ and $W|B =U_2|D_2$ in \eqref{eq01}, one gets  
 $$
 \mu^\ast(Z|C\sqcap(U_1|D_1\sqcup U_2|D_2))=\mu^\ast(Z|C \sqcap U_1|D_1)+\mu^\ast(Z|C \sqcap U_2|D_2) \quad \text{ for all } Z|C\in\P(X). 
 $$
 By induction, 
\begin{equation}\label{eq:000}
 \mu^\ast(Z|C\sqcap (\sqcup_{k\leq n}(U_k|D_k)))=\sum_{k\leq n}\mu^\ast(Z|C\sqcap U_k|D_k) \quad \text{for all }Z|C\in \P(X). 
 \end{equation}
 From the previous we know that $Y_n|E_n=\sqcup_{k\leq n}(U_k|D_k)\in \X(\mu^\ast)$.  
 By (A2), $Z|C\sqcap (U|D)^\sqsubset\sqsubseteq Z|C\sqcap (Y_n|E_n)^\sqsubset$ implies $\mu^\ast(Z|C\sqcap (U|D)^\sqsubset)\leq \mu^\ast(Z|C \sqcap  (Y_n|E_n)^\sqsubset)$ for all $Z|C\in\P(X)$.  
 Therefore, we obtain from \eqref{eq:000}
 $$
 \mu^\ast(Z|C)=\mu^\ast(Z|C \sqcap Y_n|E_n)+\mu^\ast(Z|C\sqcap(Y_n|E_n)^\sqsubset)\geq \sum_{k \leq n}\mu^\ast(Z|C \sqcap U_k|D_k)+\mu^\ast(Z|X\sqcap (U|D)^\sqsubset) 
 $$
 for all $Z|C\in \P(X)$ and $n\in \N$. 
 By (A3), 
 \begin{equation}\label{eq02}
 \mu^\ast(Z|C)\geq \sum_{k\geq 1}\mu^\ast(Z|C\sqcap U_k|D_k )+\mu^\ast(Z|C\sqcap (U|D)^\sqsubset)\geq \mu^\ast(Z|C\sqcap U|D)+\mu^\ast(Z|C \sqcap (U|D)^\sqsubset) 
 \end{equation}
 for all $Z|C \in \P(X)$,  which means $U|D\in\mathcal{X}(\mu^\ast)$. 
 By Theorem \ref{pilambda}, $\mathcal{X}(\mu^\ast)$ is a stable $\sigma$-algebra. 
To see that $\mu^\ast:\mathcal{X}(\mu^\ast)\to \bar L^0_+$ is a stable measure, replace $Z|C$ by $U|D$ in \eqref{eq02} and note that the reverse inequality follows from (A3). 
 \end{proof}
This leads us to a conditional version of Carath\'eodory's extension theorem. 
\begin{theorem}\label{t:caratheodory}
Let $X$ be a stable set of functions on $\Omega$, $\mathcal{X}$ be a stable ring on $X$ and $\mu:\mathcal{X}\to\bar L^0_+$ be a stable pre-measure. 
Then there exists a stable measure $\nu\colon\Sigma(\mathcal{X})\to \bar L^0_+$ which coincides with $\mu$ on $\mathcal{X}$. 
\end{theorem}
\begin{proof}
For each $V|A\in \P(X)$, let
\begin{equation*}
B_{V|A}:=\sup\{B^\prime\in \F\colon \text{ there is a sequence }(U_k|C_k) \text{ in } \mathcal{X} \text{ with } (V|A)|B^\prime \sqsubseteq \sqcup_k U_k|C_k\}.  
\end{equation*}
By an exhaustion argument, $B_{V|A}$ is attained. 
Let $\mathcal{U}(V|A)$ be the collection of all sequences $(U_k|C_k)$ in $\X$ such that $(V|A)|B_{V|A} \sqsubseteq \sqcup_k U_k|C_k$. 
By (S3) and the stability of $\X$ it follows that $\mathcal{U}(V|A)$ is a stable collection. 
Moreover, one has $B_{\sum (V_k|A_k)|D_k}=\cup_k (B_{V_k|A_k}\cap D_k)$ for all sequences $(V_k|A_k)$ in $\P(X)$ and partitions $(D_k)$ of $\Omega$. 
Thus $\mu^\ast\colon \P(X)\to\bar L^0_+$ defined by 
\[
 \mu^\ast(V|A):=\inf\{\sum_{k\geq 1} \mu(U_k|C_k)\colon (U_k|C_k)\in \mathcal{U}(V|A)\}|B_{V|A} + \infty|B_{V|A}^c 
\]
is a well-defined stable set function. 
We want to show that $\mu^\ast$ is a stable outer measure. 
Properties (A1) and (A2) are easy to check. 
 As for (A3), let $(V_k|A_k )$ be a sequence in $\P(X)$. 
 Clearly, $B:=\cup_k B_{V_k|A_k}=B_{\sqcup_k (V_k|A_k)}$. 
 Fix $\eps >0$. 
 For each $k$, let $(U_{k,n}|C_{k,n})\in \mathcal{U}(V_k|A_k)$ be such that 
 \begin{equation*}
  \sum_{n\geq 1}\mu(U_{k,n}|C_{k,n}) \leq \mu^\ast(V_k|A_k )+ 2^{-\eps} \quad \text{on }B. 
 \end{equation*}
Then 
 \begin{equation*}\label{eq:8}
  \mu^\ast(\sqcup_{k} V_k|A_k ) \leq \sum_{n,k\geq 1}\mu^\ast(U_{k,n}|C_{k,n}) \leq \sum_{k\geq 1} \mu^\ast(V_k|A_k )+\eps \quad \text{on } B, 
 \end{equation*}
which proves (A3). 

 Let $V|A\in\mathcal{X}$. We want to show that $V|A\in \X(\mu^\ast)$, that is, 
 \begin{equation}\label{eq03}
  \mu^\ast(Z|C\sqcap V|A)+\mu^\ast(Z|C\sqcap (V|A)^\sqsubset)\leq \mu^\ast(Z|C), \quad \text{ for all } Z|C\in\P(X). 
 \end{equation}
 Let $Z|C\in \P(X)$. On $B_{Z|C}^c$ there is nothing to show. 
 On the other hand, we have $B_{Z|C\sqcap V|A}\cup B_{Z|C\sqcap (V|A)^\sqsubset}\subseteq B_{Z|C}$. 
 Moreover,  for $(U_k|C_k)\in \mathcal{U}(Z|C)$ it follows from (M2) that 
 \begin{equation*}
  \sum_{k\geq 1}\mu(U_k|C_k)=\sum_{k\geq 1} \mu(U_k|C_k \sqcap V|A)+\sum_{k\geq 1}\mu(U_k|C_k \sqcap (V|A)^\sqsubset). 
 \end{equation*}
 Hence  \eqref{eq03} is also satisfied on $B_{Z|C}$. 
We have shown $\mathcal{X}\subseteq \mathcal{X}(\mu^\ast)$. 
Now it follows from Proposition \ref{p:outermeas} that the restriction of $\mu^\ast$ to $\Sigma(\X)$, denoted by $\nu$, is a stable measure.  
Since $B_{V|A}=\Omega$ for $V|A\in \X$, and by (M4) and (M6), 
\[
\mu(V|A)=\mu(\sqcup_k (U_k|C_k\sqcap V|A))\leq \sum_{k} \mu(U_k|C_k\sqcap V|A)\leq \sum_k \mu(U_k|C_k),
\]
for all $(U_k|C_k)\in \mathcal{U}(V|A)$, it follows that $\nu$ coincides with $\mu$ on $\X$. 
\end{proof}
As for the uniqueness of the previous extension, we have the following result. 
\begin{proposition}\label{p:uniqueness}
Let $(X,\mathcal{X})$ be a stable measurable space and $\mathcal{E}$ a stable generator of $\mathcal{X}$ which is closed under finite conditional intersections. 
For two  stable measures $\mu,\nu$ on $(X,\X)$, suppose $\mu(V|A)=\nu(V|A)$ for all $V|A\in\mathcal{E}$ and there exists a sequence $(Z_k|C_k)$ in $\mathcal{E}$ with $\sqcup_k Z_k|C_k=X$ and $\mu(Z_k|C_k)=\nu(Z_k|C_k)<\infty$ for all $k$. Then $\mu=\nu$.  
\end{proposition}
 \begin{proof}
 The proof follows from a monotone class argument in the present context. 
 Indeed, for $V|A\in \mathcal{E}$ with $\mu(V|A)=\nu(V|A)<\infty$, let 
 \begin{equation*}
  \mathcal{D}_{V|A}=\{W|B\in \mathcal{X}\colon \mu(V|A\sqcap W|B)=\nu(V|A\sqcap W|B)\}. 
  \end{equation*} 
  Then $X\in \D_{V|A}$. Stability of $\D_{V|A}$ follows from (S3). 
  By (M5) and Boolean arithmetic, $\D_{V|A}$ is closed under complementation. 
  By (M2), $\D_{V|A}$ contains the conditional union of every pairwise disjoint sequence of its elements. 
  We have that $\D_{V|A}$ is a stable Dynkin system with $\mathcal{E}\subseteq \D_{V|A}$. 
  By Theorem \ref{pilambda}, $\mathcal{D}_{V|A}=\mathcal{X}$. Hence, 
  \begin{equation*}
   \mu(W|B\sqcap V|A)=\nu(W|B\sqcap V|A),
  \end{equation*}
  for all $W|B\in\mathcal{X}$ and $V|A\in\mathcal{E}$ with $\mu(V|A)=\nu(V|A)$. 
  Therefore, 
 \begin{equation*}
  \mu(W|B  \sqcap Z_k|C_k)=\nu( W|B \sqcap Z_k|C_k)
 \end{equation*}
 for all $k$, and the claim follows from (M2). 
 \end{proof}
  Fix a stable measure space $(X,\X,\mu)$ in the remainder of this section. 
We construct a conditional Lebesgue integral for  stably measurable functions $f\colon X\to L^0$, where we endow $L^0$ with the stable Borel $\sigma$-algebra $\B(L^0)$ defined in Example \ref{exp:algebra}.2). 
Consider the following stable generators of $\B(L^0)$. 
For $r\in L^0$, let 
\begin{align*} 
[r,\infty[&:=\{s\in L^0\colon r\leq s< \infty\},\quad ]r,\infty[:=\{s\in L^0\colon r< s< \infty\},\\
]-\infty, r]&:=\{s\in L^0\colon -\infty < s\leq r\},\quad ]-\infty, r[:=\{s\in L^0\colon -\infty< s< r\}. 
\end{align*}
Notice that $[r,\infty[^\sqsubset= ]-\infty, r[$ and $]r,\infty[^\sqsubset=]-\infty, r]$. 
The collection of all of each type of these intervals is a stable collection in $\P(L^0)$, and by Boolean arithmetic, a generator of $\B(L^0)$. 

The concatenation of a sequence of  stably measurable functions $f_k\colon X\to L^0$ along a partition $(A_k)$ of $\Omega$ is defined in \eqref{eq:concatenationfct}, and it is a  stably measurable function. 
The sum $f+g$ and product $f\cdot g$ of two  stably measurable functions $f,g\colon X\to L^0$ is defined pointwise, and it can be checked that they are  stably measurable functions.  
Further, we write $f\leq g$ whenever $f(x)\leq g(x)$ in $L^0$ for all $x\in X$. 
It follows that $\max\{f,g\}$ and $\min\{f,g\}$ are stably measurable.  
The convergence of a sequence of stably measurable functions $f_k\colon X\to L^0$  to a function $f\colon X\to L^0$ is defined by 
\begin{equation}\label{eq:convergence}
f_k(x)\to f(x) \text{ a.s.}
\end{equation}
for all $x\in X$, if this limit exists, in which case $f$ is stably measurable.  

We introduce stable indicator functions and  stable elementary  functions. 
Let $V|A\in \P( X)$, and for each $x\in X$ let 
\[
 A_x:= \sup \{A^\prime\in \F \colon A^\prime \subseteq A, x|A^\prime \in V|A^\prime\}. 
\]
By an exhaustion argument, $A_x$ is attained. 
Further, $A_{\sum x_k|A_k}=\cup_k (A_{x_k}\cap A_k)$ for every sequence $(x_k)$ in $X$ and each partition $(A_k)$ of $\Omega$. 
Thus the function $1_{V|A}\colon X\to L^0$ defined by $x\mapsto 1|A_x + 0|A_x^c$ is well-defined and stable, and called the \emph{stable indicator function of $V|A$}.   
The following properties of stable indicator functions can be directly checked from the definition. 
\begin{itemize}
\item[(D1)] $1_{V|A}=1_{V}|A+0|A^c$ for all $V|A\sqsubseteq X$; 
\item[(D2)] $1_{\sum (V_k|A_k)|B_k}=\sum 1_{V_k|A_k}|B_k$ for all sequences $(V_k|A_k)$ in $\P(X)$ and partitions $(B_k)$ of $\Omega$;  
\item[(D3)] $1_{\sqcup_{i} V_k|A_k}=\sum_{k} 1_{V_k|A_k}$ for all sequences $(V_k|A_k)$ of  pairwise disjoint elements in  $\P(X)$; 
\item[(D4)] $1_{V|A}=1-1_{(V|A)^\sqsubset}$ for all $V|A\sqsubseteq X$. 
\end{itemize}
\begin{definition}
Let $(r_k)_{k\leq n}$ be a finite family in $L^0_+$ and $(V_k|A_k)_{k\leq n}$ a finite family of pairwise disjoint elements in $\X$ with $\sqcup_k V_k|A_k=X$. 
The function $\sum_{k\leq n} r_k 1_{V_k|A_k}\colon X\to L^0_+$ is called a \emph{stable elementary function}. 
\end{definition}
Without further notice, we identify $\N$ with a subset of $L^0_s(\N)$ by the embedding $n\mapsto n1_\Omega$ and $\R$ with a subset of $L^0$ by the embedding $r\mapsto r1_\Omega$. Also $L^0_s(\N)$ can be understood as a subset of $L^0$. 
\begin{remark}
We can define $r \cdot s$ and $r+s$ for arbitrary $r,s\in \bar L^0_+$ by considering the conventions $a\cdot \infty=\infty \cdot a=\infty$ for all $a\in \mathbb{R}_{++}$, $a+ \infty=\infty + a=\infty$ for all $a\in \mathbb{R}_{+}$ and $0\cdot \infty =\infty \cdot 0=0$ in a pointwise a.s.~sense. 
\end{remark}
\begin{definition}
For a  stable elementary  function $f=\sum_{k\leq n} r_k 1_{V_k|A_k}$, we define its \emph{stable Lebesgue integral} as 
\begin{equation*}\label{eq:int1}
\int_X f d\mu:=\sum_{k\leq n}  r_k \mu(V_k|A_k)\in \bar L^0_+. 
\end{equation*}
\end{definition}
To extend the previous definition to  stably measurable functions, we need the following two lemmas. 
\begin{lemma}\label{lemma100}
Let $f\colon X\to L^0_+$ be a stable elementary function and $(f_n)$ an increasing sequence of stable elementary functions $f_n\colon X\to L^0_+$ such that $f\leq \sup_n f_n$. 
Then it holds $\int_X f d\mu \leq \sup_n \int_X f_n d\mu$. 
\end{lemma} 
\begin{proof}
Suppose $f=\sum_{k\leq m} r_k 1_{V_k|A_k}$. 
Fix $r\in L^0$ with $0<r<1$. 
For every $n$, let $W_n|B_n:=(f_n-r f)^{-1}([0,\infty[)$. 
Then $f_n\geq r f 1_{W_n|B_n}$, and therefore 
\begin{equation}\label{eq800}
\int_X f_n d\mu\geq r\int_X f 1_{W_n|B_n} d\mu
\end{equation}
  for each $n$. 
Since $(f_n)$ is an increasing sequence with $f\leq \sup_n f_n$, it follows that $(W_n|B_n)$ is an increasing sequence with $X=\sqcup_n W_n|B_n$. 
Thus $(W_n|B_n\sqcap V_k|A_k)$ is an increasing sequence with $V_k|A_k=\sqcup_n (W_n|B_n\sqcap V_k|A_k)$ for each $k$. 
By (M7), 
\[
\int_X f d\mu=\lim_{n} \int_X f 1_{W_n|B_n} d\mu. 
\]
Then from \eqref{eq800}
\[
\sup_n \int_X f_n d\mu\geq \sup_n r \int_X f 1_{W_n|B_n}d\mu= r\lim_n \int_X f 1_{W_n|B_n}d\mu= r \int_X f d\mu. 
\]
Since $r\in L^0$ with $0<r<1$ is arbitrary, the claim follows. 
\end{proof}
\begin{lemma}\label{l:approxbysimplefct}
For every  stably measurable function $f\colon X\to L^0_+$ there exists an increasing sequence $(f_n)$ of stable elementary functions such that $\sup_n f_n(x)=f(x)$ for all $x\in X$. 
\end{lemma}
\begin{proof}
We denote $\{f< r\}:=f^{-1}(]-\infty,r[)$ and $\{r\leq f\}:=f^{-1}([r,\infty[)$,  $r\in L^0$.  
For each $n\in \N$, let 
\[
W_{k,n}|A_{k,n}:=
\begin{cases}
\{k2^{-n}\leq f\}\sqcap \{f<(k+1)2^n\}, \quad 0\leq k\leq n 2^n-1, \\
\{f\geq n\}, \quad k=n2^n. 
\end{cases}
\]
Define
\[
f_n:=\sum_{1\leq k\leq n2^n} k 2^{-n} 1_{W_{k,n}|A_{k,n}}. 
\]
Fix $x\in X$ and $n\in \N$. 
For $k=0,\ldots,n2^n-1$, we have 
\begin{align*}
A_k&=\sup\{A^\prime\in\F\colon A^\prime\subset A_{k,n}, x|A^\prime\in W_{k,n}|A^\prime\}\\
&=\sup\{A^\prime\in\F\colon A^\prime\subset A_{2k,n+1}, x|A^\prime\in W_{2k,n+1}|A^\prime \text{ or }A^\prime\subset A_{2k+1,n+1}, x|A^\prime\in W_{2k+1,n+1}|A^\prime\}, 
\end{align*}
where both suprema are attained by an exhaustion argument. 
Let 
\[
A_{n 2^n}=\sup\{A^\prime \in\F\colon A^\prime\subset A_{n2^n,n}, x|A\in W_{n2^n,n}|A^\prime\}. 
\]
Since $(W_{k,n}|A_{k,n})_{0\leq k\leq n 2^n}$ is a partition of $X$, $(A_k)_{0\leq k\leq n 2^n}$ is a partition of $\Omega$. 
Hence from $f_n(x)\leq f_{n+1}(x)$ on $A_k$ for all $0\leq k \leq n2^n$ it follows that $f_n(x)\leq f_{n+1}(x)$. 
Thus $(f_n)$ is increasing. 
By construction, $\sup_n f_n=f$. 
\end{proof}
Since $L^0$ is a vector lattice, every  stably measurable function $f\colon X\to L^0$ can be written as the difference $f^+-f^-$ of the  stably measurable functions $f^+:=\max\{f,0\}$ and $f^-:=\max\{-f(x),0\}$.  
It follows from the previous two lemmas that the following is well-defined. 
\begin{definition}
Let $f\colon X\to L^0_+$ be  stably measurable. 
Define 
\begin{equation*}\label{eq:int2}
 \int_X f d\mu:= \sup_n \int_X f_n d\mu\in \bar L^0_+,
\end{equation*}
where $(f_n)$ is an increasing sequence of  stable elementary  functions with $\sup_n f_n=f$. 
Let  $f\colon X\to L^0$ be  stably measurable. 
We say that $f$ is \emph{integrable} whenever $\int_X f^+ d\mu,\int_X f^- d\mu\in L^0$. 
In this case, we define its \emph{stable Lebesgue integral} by 
\begin{equation*}\label{eq:int3}
\int_X f d\mu:=\int_X f^+ d\mu - \int_X f^- d\mu. 
\end{equation*}
\end{definition}
The following properties of the stable integral can be verified directly:  
 \begin{itemize}
 \item[(I1)] $\int_X \sum f_k|A_k d\mu=\sum \int_X f_k d\mu|A_k$ for every sequence $(f_k)$ of  integrable functions and every partition $(A_k)$ of $\Omega$.    
 \item[(I2)] $\int_X f d\mu \leq \int_X g d\mu$ for every pair of  integrable functions $f,g$ with $f\leq g$. 
 \item[(I3)] $\int_X r f + g d\mu= r \int_X f d\mu + \int_X g d\mu$ for all  integrable functions $f,g$ and every $r\in L^0$.  
 \end{itemize}
We have the following version of the monotone convergence theorem. 
\begin{theorem}\label{t:monotone convergence}
Let $(f_n)$ be an increasing sequence of integrable functions $f_n\colon X\to L^0_+$. Then it holds 
\begin{equation*}
\int_X \sup_n f_n d\mu = \sup_n \int_X f_n d\mu. 
\end{equation*}
In particular, 
\[
\int_X \left(\sum_{n=1}^\infty f_n\right)d\mu=\sum_{n=1}^\infty \int_X f_n d\mu.  
\]
\end{theorem}
\begin{proof}
Let $f:=\sup_n f_n$. 
By (I2), it is enough to build a sequence $(h_n)$ of  stable elementary  functions with $\sup_n h_n=f$ and $h_n\leq f_n$ for all $n$. 
By Lemma \ref{l:approxbysimplefct}, for each $f_n$ there exists a sequence $(g_{m,n})$ of  stable elementary  functions such that $f_n=\sup_m g_{m,n}$. 
Then $h_m:=\max_{k\leq m} g_{m,k}$ satisfies the required. 
\end{proof}
\section{Kernels, conditional distributions and stable measures}\label{s:kernel}
In this section, we establish a link between kernels and stable measures. 
Based thereupon, we can extend the representation of conditional distribution of random variables with values in a not necessarily standard Borel space by replacing probability kernels with stable probability measures. 
We apply this representation to compute the conditional expectation of functions of random variables by means of the stable Lebesgue integral. 

Unless mentioned otherwise, $(E,\mathscr{E})$ denotes an arbitrary measurable space throughout this section. 
Recall that $L^0_m(E)$ denotes the space of all measurable functions $x\colon \Omega\to E$, which are called random variables with values in $E$. 
Let $\mathscr{R}$ be a classical ring of sets generating the $\sigma$-algebra $\mathscr{E}$. 
For every sequence $(F_k)$ in $\mathscr{R}$ and each partition $(A_k)$, let 
\begin{equation}\label{eq:ring1}
\sum L^0_m(F_k)|A_k:=\{x\in L^0_m(E)\colon x|A_k \in F_k \text{ a.s.~for all }k\}. 
\end{equation}
By definition, $\sum L^0_m(F_k)|A_k$ is a stable subset of $L^0_m(E)$. Since $\mathscr{R}$ is a ring, by a straightforward computation, 
\begin{equation}\label{eq:ring}
\mathcal{R}:=\left\{(\sum L^0_m(F_k)|A_k)|A\colon (F_k) \text{ sequence in }\mathscr{R}, (A_k) \text{ partition of }\Omega, A\in\F\right\}
\end{equation}
is a stable ring on $L^0_m(E)$. 
We consider the stable $\sigma$-algebra $\Sigma(\mathcal{R})$ generated by $\mathcal{R}$. 
If we replace in \eqref{eq:ring} the ring $\mathscr{R}$ by the larger $\sigma$-algebra $\mathscr{E}$, then the stable $\sigma$-algebra generated by the stable ring with respect to $\mathscr{E}$ coincides with the one generated by the stable ring with respect to $\mathscr{R}$.  
Indeed, this follows from the observation 
\begin{align*}
\sqcup_{n} L^0_m(F_n)&=L^0_m(\cup_n F_n), \\
L^0_m(F)^\sqsubset&=L^0_m(F^c),
\end{align*}
for all $(F_n)$ and $F$ in $\mathscr{E}$.  
In particular, if $E$ is a separable topological space, by Boolean arithmetic, the stable measurable space $(L^0_m(E),\Sigma(\mathcal{R}))$ is identical with the stable Borel space $(L^0(E),\B(L^0(E)))$, see Examples \ref{exp:algebra}.2) and Remark \ref{r:sep}. 

Recall that a \emph{kernel} on $E$ is a function $\kappa\colon \Omega\times \mathscr{E}\to \overline \R_+ $ such that $\kappa(\omega,F)$ is $\F$-measurable in $\omega\in \Omega$ for fixed $F\in \mathscr{E}$ and a measure in $F\in\mathscr{E}$ for fixed $\omega\in \Omega$. 
A kernel $\kappa$ is said to be a probability kernel if $\kappa(\omega,F)$ is a probability measure in $F\in \mathscr{E}$ for all $\omega\in \Omega$. 
We have the following main result which connects kernels with stable measures. 
\begin{theorem}\label{t:kernel}
\begin{itemize}[fullwidth]
\item[(i)]
For every kernel $\kappa\colon \Omega\times \mathscr{E}\to \overline\R_+$ there exists a stable measure on $(L^0_m(E),\Sigma(\mathcal{R}))$, denoted by $\mu_\kappa$, such that 
\begin{equation*}\label{eq:repkern}
\mu_\kappa(L^0_m(F))(\omega)=\kappa(\omega,F) \text{ a.s.~for all }F\in \mathscr{E}. 
\end{equation*}
If $\kappa$ is a probability kernel, then $\mu_\kappa$ is the unique stable probability measure satisfying the previous equation. 
\item[(ii)] Suppose $(E,\mathscr{E})$ is a standard Borel space. For every stable probability measure $\mu$ on the stable Borel space $(L^0(E),\B(L^0(E)))$ there exists a probability kernel on $E$, denoted by $\kappa_\mu$, such that 
\begin{equation*}\label{eq:repkern}
\mu(L^0(F))(\omega)=\kappa_\mu(\omega,F) \text{ a.s.~for all }F\in \mathscr{E}. 
\end{equation*}
In particular, one has the following reciprocality identities: 
\begin{align*}
\mu_{\kappa_\mu}(V|A)(\omega)&=\mu(V|A)(\omega) \text{ a.s.~for all } V|A\in \B(L^0(E)), \\ 
\kappa_{\mu_\kappa}(\omega,F)&=\kappa(\omega,F) \text{ a.s.~for all }F\in \mathscr{E}. 
\end{align*}
\end{itemize}
\end{theorem}
\begin{proof}
\begin{itemize}[fullwidth]
\item[(i)] For $(\sum L^0_m(F_k)|A_k)|A\in \mathcal{R}$, define 
\begin{equation}\label{eq:kerntosmeas}
\mu_\kappa((\sum L^0_m(F_k)|A_k)|A):=(\sum \kappa(\cdot,F_k)|A_k)|A + 0|A^c. 
\end{equation}
If $\mu_\kappa$ is a stable pre-measure, then the first claim follows from Theorem \ref{t:caratheodory}. 
(M1) is satisfied by definition. 
As for (M2), let $((\sum L^0_m(F_{k,n})|A_{k,n})|A_n)$ be a sequence of pairwise disjoint elements in $\mathcal{R}$ such that $\sqcup_n (\sum L^0_m(F_{k,n})|A_{k,n})|A_n\in \mathcal{R}$. 
Denoting $A:=\cup_n A_n$, the conditional union must be of the form 
\[
\sqcup_n (\sum L^0_m(F_{k,n})|A_{k,n})|A_n=(\sum L^0_m(F_k)|A_k)|A. 
\]
Since the above sequence is pairwise disjoint, each $F_k$ is the disjoint union of all $F_{k,n}$ with $\emptyset \neq A\cap A_k\subseteq A_n\cap A_{k,n}$. 
Thus (M2) follows from the pointwise $\sigma$-additivity of the kernel $\kappa$. 
The second claim follows immediately from Proposition \ref{p:uniqueness}. 
\item[(ii)] The proof can be carried out similarly to the one of the existence of regular conditional distributions. 
We will follow the main arguments in \cite[Theorem 6.3]{Kallenberg2002}. 

By \cite[Theorem A1.2]{Kallenberg2002}, there exists a Borel isomorphism from $E$ to a Borel subset  $S$ of $\mathbb{R}$. 
Therefore, it is enough to prove the claim for the Borel space $(S,\mathscr{B}(S))$. 
For each $q\in \mathbb{Q}$, let $f_q=f(\cdot,q)\colon \Omega \to [0,1]$ be defined by 
\begin{equation}\label{eq:3}
f(\cdot,q)=\mu(L^0(]-\infty,q]))\text{ a.s.} 
\end{equation}
By (M4), one has $f(\cdot, p)\leq f(\cdot,q)$ whenever $p\leq q$. Let $A$ be the set of all $\omega\in \Omega$ such that $f(\omega,q)$ is  increasing in $q\in \mathbb{Q}$ with limits  $1$ and $0$ at $\pm\infty$. 
Since $A$ is specified by countably many measurable conditions each of which holds a.s., we have $A\in\F$.  
Define 
\[
 F(\omega,x):=1_{A}(\omega)\inf_{q>x}f(\omega,q)+ 1_{A^c}(\omega)1_{\{x\geq 0\}}, \quad x\in \mathbb{R}, \, \omega\in \Omega. 
\]
From (M4) and (M8) it follows that $F(\omega,\cdot)$ is a distribution function for every $\omega\in \Omega$. 
Hence, by \cite[Proposition 2.14]{Kallenberg2002}, there exists a probability measure $\kappa(\omega,\cdot)$ such that $\kappa(\omega,]-\infty,x])=F(\omega,x)$ for all $x\in \mathbb{R}$ and $\omega\in \Omega$.
By a monotone class argument, $\kappa$ is a probability kernel. 
Using a monotone class argument based on an a.s.~interpretation of (M5) and (M7) shows that $\kappa(\cdot,F)=\mu(L^0_m(F))$ for all $F\in\mathscr{B}(\R)$. 
In particular, $\kappa(\cdot,S^c)=0$ a.s., and thus $\kappa_\mu(\cdot,F)=\mu(L^0_m(F))$ for all $F\in\mathscr{B}(S)$, where $\kappa_\mu$ is the probability kernel defined by 
\[
\kappa_\mu(\omega,\cdot):=\kappa(\omega,\cdot)1_{\{\kappa(\omega,S)=1\}}+\delta_s1_{\{\kappa(\omega,S)<1\}}, 
\]
where $s\in S$ is arbitrary. 
Uniqueness follows from \eqref{eq:3} by a monotone class argument. 
The reciprocality identities are an immediate consequence of the constructions of $\kappa_\mu$ and $\mu_\kappa$. 
\end{itemize}
\end{proof}
The first part of the previous theorem provides a procedure to construct stable measures from classical ones as follows. 
\begin{examples}
Let $\lambda$ be a $\sigma$-finite measure on $(E,\mathscr{E})$, and view $\lambda$ as a constant kernel on $\Omega\times \mathscr{E}$. 
Let $\mu_\lambda$ be the induced stable measure on $(L^0_m(E),\Sigma(\mathcal{R}))$ as in Theorem \ref{t:kernel}. 
Let $(F_n)$ be an increasing sequence in $\mathscr{E}$ such that $\cup_n F_n=E$ and $\lambda(F_n)<\infty$ for all $n$. 
Then $\sqcup_n L^0_m(F_n)=L^0(E)$ and $\mu_\lambda(L^0_m(F_n))<\infty$ for all $n$, which implies that $\mu_\lambda$ is a $\sigma$-finite stable measure. 
By Proposition \ref{p:uniqueness}, $\mu_\lambda$ is the unique stable measure extending the classical measure $\lambda$ from $(E,\mathscr{E})$ to $(L^0_m(E),\Sigma(\mathcal{R}))$. 
In particular, if $(\R^d ,\mathscr{B}(\R^d),\lambda)$ is the classical Borel space with $\lambda$ the $d$-dimensional Lebesgue measure, then the extension $\mu_\lambda$ to the stable Borel space $((L^0)^d,\B((L^0)^d))$ is called the \emph{stable $d$-dimensional Lebesgue measure}. 
\end{examples}
Next, we study the link to the conditional distribution of random variables. 
For the remainder of this section, let $\G\subseteq\F$ be a sub-$\sigma$-algebra and $\xi:\Omega\to E$ a random variable. 
Recall that a regular conditional distribution of $\xi$ given $\G$ is a version of the function $\mathbb{P}[\xi \in \cdot|\G]$ on $\Omega\times \mathscr{E}$ which is a probability kernel. 
It is well known that such a representing probability kernel exists whenever $(E,\mathscr{E})$ is a standard Borel space, see e.g.~\cite[Theorem 6.3]{Kallenberg2002}. 
In view of Theorem \ref{t:kernel}, we can extend this representation result to arbitrary spaces as follows. 
\begin{example} 
All objects will be defined relative to the probability space $(\Omega,\G,\mathbb{P})$, e.g.~$\G$-stable sets, $\G$-stable conditional power sets, $\G$-stable conditional set operations, etc. 
Let $L^0_{m,\G}(E)$ denote the space of all $\G$-measurable functions $x\colon \Omega\to E$. 
For a sequence $(F_k)$ in $\mathscr{E}$, a partition $(A_k)$ with $A_k\in \G$ for all $k$, define 
\[
\sum L^0_{m,\G}(F_k)|A_k:=\{x\in L^0_{m,\G}(E)\colon x|A_k(\omega)\in F_k \text{ a.s.~for all } k\}.
\]
Let $\mathcal{R}_\G$ be the $\G$-stable ring on $L^0_{m,\G}(E)$ consisting of all $(\sum L^0_{m,\G}(F_k)|A_k)|A$, where $(F_k)$ is a sequence in $\mathscr{E}$, $(A_k)$ is a partition with $A_k\in \G$ for all $k$ and $A\in\G$. 
We denote by $\Sigma_\G(\mathcal{R}_\G)$ the smallest $\G$-stable $\sigma$-algebra on $L^0_{m,\G}(E)$ including $\mathcal{R}_\G$. 
Then there exists a unique $\G$-stable probability measure on $(L^0_{m,\G}(E),\Sigma_\G(\mathcal{R}_\G))$, denoted by $\nu=\nu_{\xi,\G}$,  such that 
\begin{equation}\label{eq:repkern2}
\mathbb{P}[\xi\in F|\G](\omega)=\nu(L^0_{m,\G}(F))(\omega) \text{ a.s.~for all }F\in \mathscr{E}. 
\end{equation}
Indeed, put 
\[
\nu((\sum L^0_{m,\G}(F_k)|A_k)|A):=(\sum \mathbb{P}[\xi \in F_k|\G]|A_k)|A + 0|A^c, \quad (\sum L^0_{m,\G}(F_k)|A_k)|A\in \mathcal{R}_\G, 
\]
and proceed similarly to the first part of the proof of Theorem \ref{t:kernel}. 
\end{example}
We want to compute the conditional expectation of $f(\xi)$ with the help of the stable Lebesgue integral on applying the representation \eqref{eq:repkern2}. 
Fix a Borel measurable function $f\colon E\to \mathbb{R}$ such that $\mathbb{E}[|f(\xi)|]<\infty$. 
Recall that the distribution of $\xi$ is defined by the pushforward measure $\mathbb{P}_\xi:=\mathbb{P}\circ \xi^{-1}$ on $(E,\mathscr{E})$. 
By the transformation theorem, it holds 
\begin{equation}\label{eq:factorization}
\mathbb{E}[f(\xi)]=\int_E f d \mathbb{P}_\xi. 
\end{equation}
If $(E,\mathscr{E})$ is a standard Borel space, then we have the following conditional analogue of \eqref{eq:factorization}:  
\begin{equation}\label{eq:intrep}
\mathbb{E}[f(\xi)|\G](\omega)=\int f(x)\kappa_\xi(\omega,dx) \text{ a.s.},
\end{equation}
where $\kappa_\xi$ is a regular conditional distribution of $\xi$ given $\G$. 
We can extend the representation \eqref{eq:intrep} to spaces which are not necessarily standard Borel in the following way. 
\begin{theorem}\label{p:conddistrep}
Let $f\colon E\to\R$ be a Borel measurable function with $\mathbb{E}[|f(\xi)|]<\infty$. Then $f$ extends to a $\G$-stable integrable function $\hat{f}\colon L^0_{m,\G}(E)\to L^0_\G$ such that 
\begin{equation}\label{eq:intrepext}
\mathbb{E}[f(\xi)|\G](\omega)=\left(\int_{L^0_{m,\G}(E)} \hat{f} d\nu_{\xi,\G}\right)(\omega) \text{ a.s. }
\end{equation} 
\end{theorem}
\begin{proof}
By inspection, $\hat{f}\colon L^0_{m,\G}(E)\to L^0_\G$ given by $\hat{f}(x):=f\circ x$ is well-defined and $\G$-stable. 
We show that $\hat{f}$ satisfies \eqref{eq:intrepext}. 
We can assume w.l.o.g.~that $f\geq 0$. 
First, suppose $f=\sum_{k\leq n} x_k 1_{F_k}$. 
Then $\hat{f}$ is equal to the $\G$-stable elementary  function $\sum_{k\leq n} x_k 1_{L^0_{m,\G}(F_k)}$. 
It follows from (I3) and \eqref{eq:repkern2} that 
\begin{equation}\label{eq:simple}
\mathbb{E}[f(\xi)|\G](\omega)=\sum_{k\leq n} x_k \nu(L^0_{m,\G}(F_k))(\omega)=\left(\int_{L^0_{m,\G}(E)} \hat{f} d\nu_{\xi,\G}\right)(\omega) \text{ a.s. }
\end{equation}
Second, suppose $f\geq 0$.
Let $(f_n)$ be an increasing sequence of simple functions such that $f=\lim f_n$.  
Then $(\hat{f}_n)$ is an increasing sequence of  stable elementary  functions such that $\hat{f}$ is the pointwise limit of $(\hat{f}_n)$. 
By the monotone convergence theorem for conditional expectations and Theorem \ref{t:monotone convergence}, we obtain from \eqref{eq:simple}
\[
\mathbb{E}[f(\xi)|\G](\omega)=\lim_n \mathbb{E}[f_n(\xi)|\G](\omega)=\lim_n\left(\int_{L^0_{m,\G}(E)} \hat{f}_n d\nu_{\xi,\G}\right)(\omega)=\left(\int_{L^0_{m,\G}(E)} \hat{f} d\nu_{\xi,\G}\right)(\omega) \text{ a.s. }
\]
\end{proof}
\section{A conditional version of Fubini's theorem, the Radon-Nikod\'ym theorem, the Daniell-Stone theorem and the Riesz representation theorem}\label{s:theorems}
In this section, we establish four important theorems in measure theory for stable measure spaces. 
We start with Fubini's theorem and the Radon-Nikod\'ym theorem, and close with the Daniell-Stone theorem from which we derive two Riesz type represenation and regularity results. 
\subsection{Fubini's theorem.}
Throughout this subsection, let $(X,\X,\mu)$ and $(Y,\Y,\nu)$ be two stable probability spaces. 
On the stable set $X\times Y$ of functions on $\Omega$ consider the \emph{conditional rectangles} $V|A\times W|B:=V\times W|A\cap B$ where $V|A\in \X$ and $W|B\in \Y$. 
Inspection shows
\[
V_1|A_1\times W_1|B_1\sqcap V_2|A_2\times W_2|B_2=V_1|A_1\sqcap V_2|A_2 \times W_1|B_1 \sqcap W_2|B_2, 
\]
which implies that the collection 
\begin{equation}\label{eq:prodgenerator}
\mathcal{E}:=\{V|A\times W|B\colon V|A\in\mathcal{X},W|B\in\mathcal{Y}\} 
\end{equation}
is a stable collection in $\P(X\times Y)$ closed under finite conditional intersections. 
We define the \emph{stable product $\sigma$-algebra} of $\X$ and $\Y$ to be the stable $\sigma$-algebra generated by 
$\mathcal{E}$, and denote it by $\mathcal{X} \otimes \mathcal{Y}$. 
Let $V|A\sqsubseteq X\times Y$ and $x\in X$. 
The \emph{conditional $x$-section} of $V|A$ is defined to be the conditional set 
\begin{equation}\label{eq:section}
(V|A)_x:=W|D_\ast,   
\end{equation}
where   
\begin{align*}
D_\ast&:=\sup\{A^\prime\in \F\colon A^\prime\subseteq A, \text{ there is } y\in Y \text{ such that } (x,y)|A^\prime\in V|A^\prime\}, \\
W&:=\{y\in Y\colon (x,y)|D_\ast\in V|D_\ast\}. 
\end{align*}
Since $D_\ast$ is attained by an exhaustion argument, $W$ is a stable subset of $Y$ by stability of $V$. 
Similarly, we define the \emph{conditional $y$-section} of $V|A$ for some $y\in Y$. 
For a function $f\colon X\times Y\to E$, we denote by $f_x$ and $f_y$ its $x$- and $y$-section, respectively, which are defined classically as $f_x(y)=f_y(x)=f(x,y)$. 
If $f$ is a stable function, then $f_x$ and $f_y$ are stable functions as well.  
We have the following useful properties which we state for (conditional) $x$-sections; analogous properties hold for $y$-sections. 
\begin{proposition}\label{p:elementaryprop2}
Let $Z|C\in \P(X\times Y)$, $(Z_i|C_i)_{i\in I}$ and $(Z_{k}|A_k)_{k\in \N}$ be families in $\mathcal{P}(X\times Y)$, where $I$ is an arbitrary  nonempty  index set.  
Let $f\colon X\times Y\to L^0$ be a function. 
Let $x\in X$, $(x_k)$ be a sequence in $X$ and $(D_k)$ a partition. 
Then the following holds true. 
\begin{itemize}
\item[(F1)] $(\sum (Z_{k}|A_k)|D_k)_{x}=\sum (Z_{k}|A_k)_x|D_k$; 
\item[(F2)] $(Z|C)_{\sum x_k|D_k}=\sum (Z|C)_{x_k}|D_k$; 
\item[(F3)] $(\sqcup_i Z_i|C_i)_{x}=\sqcup_i (Z_i|C_i)_{x}$; 
\item[(F4)] $((Z|C)_x)^\sqsubset=((Z|C)^\sqsubset)_{x}$; 
\item[(F5)] if $Z|C\in \X\otimes \Y$, then $(Z|C)_x\in \Y$;
\item[(F6)] if $f$ is conditionally $\X\otimes \Y$-measurable, then $f_x$ is conditionally $\Y$-measurable. 
\item[(F7)] for $Z|C\in \X\otimes \Y$, the function $x\mapsto \nu((Z|C)_x)$ is conditionally $\X$-measurable. 
\end{itemize}
\end{proposition}
\begin{proof}
The statements (F1) and (F2) can easily be verified from the definitions. 
\begin{itemize}
\item[(F3)] Suppose $\sqcup_i Z_i|C_i=Z|C$ where $C=\sup_i C_i$. Let  
\begin{align*}
D_\ast&=\sup \{D^\prime\in\F\colon D^\prime\subseteq C, \text{ there is }y\in Y \text{ such that }(x,y)|D^\prime\in Z|D^\prime\},\\
D^i_\ast&=\sup \{D^\prime\in \F\colon D^\prime \subseteq C_i, \text{ there is }y\in Y \text{ such that }(x,y)|D^\prime\in Z_i|D^\prime\}. 
\end{align*}
By definition of the conditional union (see \eqref{eq:union}), one has $D_\ast \subseteq \sup_i D^{i}_\ast$. 
To see the converse inclusion, let $(D_{i_k})$ be such that $\sup_k D_{i_k}=\sup_i D_i$. 
Let $(y_{i_k})$ be a sequence in $Y$ such that $(x,y_{i_k})|D_{i_k}\in Z_{i_k}|D_{i_k}$ for all $k$. 
Let $C_1=D_{i_1}$ and $C_k=C_{i_k}\cap (C_1\cup\ldots \cup C_{i_{k-1}})^c$ for $k\geq 2$. 
For $y=\sum y_{i_k}|C_k + y_0|(\sup_k D_{i_k})^c$, where $y_0\in Y$ is arbitrary, one has $(x,y)|\sup_k D_{i_k}\in Z|\sup_k D_{i_k}$, which shows $\sup_i D_i=\sup_k D_{i_k}\subseteq D$. 
We conclude that $(Z|C)_x=\sqcup_i (Z_i|C_i)_x$. 
\item[(F4)] By (F3), 
\begin{equation}\label{eq:s10}
Y=(X\times Y)_x=(Z|C\sqcup (Z|C)^\sqsubset)_x=(Z|C)_x\sqcup ((Z|C)^\sqsubset)_x. 
\end{equation}
Suppose $(Z|C)^\sqsubset=W|B$. 
Let 
\begin{align*}
D^1_\ast&=\sup\{D^\prime\in\F\colon D^\prime\subseteq C, \text{ there is }y\in Y \text{ such that }(x,y)|D^\prime\in Z|D^\prime\}, \\
D^2_\ast&=\sup\{D^\prime\in\F\colon D^\prime\subseteq B, \text{ there is }y\in Y \text{ such that }(x,y)|D^\prime\in W|D^\prime\}. 
\end{align*}
From \eqref{eq:s10} it follows $D^1_\ast\cup D^2_\ast=\Omega$. 
Also, it holds $D_1\cap D_2=\emptyset$ since otherwise $Z|C\sqcap (Z|C)^\sqsubset\neq \{\ast\}$ which contradicts the complementation law in the Boolean algebra $\mathcal{P}(X)$. 
We conclude that the conditional union $Y=(Z|C)_x\sqcup ((Z|C)^\sqsubset)_x$ is a disjoint conditional union, and this implies $((Z|C)_x)^\sqsubset=((Z|C)^\sqsubset)_x$. 
\item[(F5)] Let $\mathcal{Z}$ be the collection of all $Z|C\in \X\otimes \Y$ such that $(Z|C)_x\in \Y$. Then $\mathcal{Z}$ includes all conditional rectangles $V|A\times W|B$ where $V|A\in \X$ and $W|B\in\Y$. 
By (F3) and (F4), $\mathcal{Z}$ is closed under conditional complementation and countable conditional unions. Since $\Y$ is a stable collection, we conclude that $\mathcal{Z}$ is a stable $\sigma$-algebra. 
Since $\mathcal{Z}$ includes the stable generator $\E$ (see \eqref{eq:prodgenerator}) and $\X\otimes \Y$ is the smallest stable $\sigma$-algebra including $\E$, it follows that $\mathcal{Z}=\X\otimes \Y$. 
\item[(F6)] By (F5), $f^{-1}_x(]r,\infty[)=(f^{-1}(]r,\infty[))_x\in \Y$ for all $r\in L^0$. 
\item[(F7)] By (F2) and the stability of $\nu$, the function $x\mapsto \nu((Z|C)_x)$ is stable. 
Let 
\begin{equation*}
\mathcal{D}=\{Z|C\in\mathcal{X}\otimes\mathcal{Y}\colon x\mapsto \nu((Z|C)_x) \text{ is stably $\X$-measurable}\}.
\end{equation*}
From $X\times Y\in \mathcal{D}$ and (F1) it follows that $\mathcal{D}$ is a stable collection. 
From (F3) and (M2), it follows that $\mathcal{D}$ is closed under taking countable conditional unions of pairwise disjoint conditional sets. 
From (F4) and (M5), $\mathcal{D}$ is also closed under conditional complementation. 
Thus $\mathcal{D}$ is a stable Dynkin system. 
Since for $V|A\in \X$ and $W|B\in \Y$ the function $x\mapsto \nu((V|A\times W|B)_x)=\nu(W|B)1_{V|A}(x)$ is stably $\X$-measurable, $\mathcal{D}=\X\otimes \Y$ due to Theorem \ref{pilambda}. 
\end{itemize}
\end{proof}
\begin{lemma}\label{p:prodmeas}
There exists a unique stable probability measure $\lambda$ on $\mathcal{X}\otimes \mathcal{Y}$ such that 
\[
\lambda(V|A\times W|B)=\mu(V|A)\nu(W|B)
\]
for all $V|A\in \mathcal{X}$ and $W|B\in\mathcal{Y}$. 
Moreover, for all $Z|C\in\mathcal{X}\otimes\mathcal{Y}$, one has  
\begin{align*}
\lambda(Z|C)=\int_X s^X_{Z|C} d\mu=\int_Y s^Y_{Z|C} d\nu, 
\end{align*}
where $s^X_{Z|C}(x):=\nu((Z|C)_x)$, $x\in X$, and $s^Y_{Z|C}(y):=\mu((Z|C)_y)$, $y\in Y$. 
\end{lemma}
We call $\mu\otimes \nu:=\lambda$ the \emph{stable product measure of $\mu$ and $\nu$} and the triple $(X\times Y, \X\otimes \Y, \mu\otimes \nu)$ the \emph{stable product probability space}. 
\begin{proof}
For $Z|C\in\mathcal{X}\otimes\mathcal{Y}$, define 
\[
\lambda(Z|C):=\int_X s^X_{Z|C} d\mu. 
\]
By (F2), the stability of $\nu$ and (I1), $\lambda$ is a stable function on $\X\otimes \Y$ satisfying (M1). 
That $\lambda$ satisfies (M2) follows from (F3) and Theorem \ref{t:monotone convergence}.  
For $V|A\in\mathcal{X}$ and $W|B\in\mathcal{Y}$, one has 
\begin{equation*}\label{eq:prod}
\lambda(V|A\times W|B)=\nu(V|A)\mu(W|B). 
\end{equation*}
By analogous arguments, the mapping 
\begin{equation*}
  \X\otimes \Y \ni Z|C\mapsto  \lambda^\prime(Z|C):=\int_Y s^Y_{Z|C} d\nu 
\end{equation*}
is a stable probability measure with the above property. 
Since $\E$ as defined in \eqref{eq:prodgenerator} is a stable generator of $\X\otimes \Y$ closed under finite conditional intersections,  the claim follows from Proposition \ref{p:uniqueness}. 
\end{proof}
We have the following version of Fubini's theorem, which can easily be extended to $\sigma$-finite stable measure spaces. 
\begin{theorem}
Let $f\colon X\times Y\to L^0$ be a stable integrable function with respect to $(X\times Y, \X\otimes \Y, \mu\otimes \nu)$. 
Then the functions 
\[
x\mapsto \int_Y f_x d\nu \quad \text{and} \quad y\mapsto \int_X f_y d\mu 
\]
are stable integrable functions, and one has 
\begin{equation*}\label{eq:51}
\int_{X\times Y} f d\mu\otimes\nu = \int_X\left(\int_Y f_x d\nu \right) d\mu =\int_Y\left(\int_X f_y d\mu \right) d\nu. 
\end{equation*}
\end{theorem}
\begin{proof}
First, suppose that $f$ is a  stable elementary  function of the form $\sum_{k\leq n} r_k 1_{Z_k|C_k}$. 
Then $f_x=\sum_{k\leq n} r_k 1_{(Z_k|C_k)_x}$ is  stably measurable due to (F6). 
By (F7), 
\begin{equation*}
x\mapsto \int_Y f_x d\nu=\sum_{k\leq n} r_k \nu((Z_k|C_k)_x) 
\end{equation*}
is  stably measurable and $L^0$-bounded. 
By Lemma \ref{p:prodmeas}, 
\begin{equation*}
 \int_X\left(\int_Y f_x d\nu\right) d\mu=\sum_{k\leq n}r_k \mu\otimes \nu(Z_k|C_k)=\int_{X\times Y} f d\mu\otimes\nu. 
\end{equation*}
Second, let $f$ be a non-negative stable integrable function, and choose with the help of Lemma \ref{l:approxbysimplefct} an increasing sequence $(f_n)$ of stable elementary functions such that $f=\sup_n f_n$. 
By Theorem \ref{t:monotone convergence}, (I2) and the first step, 
\[
\int_{X\times Y} f d\mu\otimes\nu=\int_X\left(\int_Y f_x d\nu\right) d\mu. 
\]
Finally, for an arbitrary stable integrable function $f$, the previous equality follows from the identities $|f|_x=(f_+)_x+(f_-)_x$ and $f_x=(f_+)_x-(f_-)_x$, (I3) and the first and second step.  
Analogously, one can prove 
\begin{equation*}\label{eq:51}
\int_{X\times Y} f d\mu\otimes\nu  =\int_Y\left(\int_X f_y d\mu \right) d\nu, 
\end{equation*}
which finishes the proof. 
\end{proof}
\begin{remark}
A stable function $K\colon X\times \Y\to L^0_+$ is said to be a \emph{stable Markov kernel} if 
\begin{itemize}
\item[(i)] $K(x,\cdot)$ is a stable probability measure for all $x\in X$, 
\item[(ii)] $K(\cdot,V|A)$ is stably measurable for all $V|A\in\Y$.  
\end{itemize}
We can extend Fubini's theorem to stable Markov kernels.
We provide the statement of this extension in the following, but will omit a proof as it can be worked out with a similar strategy as above.  
Let $\mu$ be a stable probability measure on $(X,\X)$.  
Then $K\otimes \mu(V|A):=\int_X K(x,(V|A)_x)d \mu$ defines a stable probability measure on $\X\otimes \Y$.  
If $f\colon X\times Y\to L^0$ is stably integrable w.r.t.~$K\otimes \mu$, then 
\[
\int_{X\times Y} f d K\otimes \mu=\int_X \int_Y f(x,y)K(x,dy)d\mu. 
\]
\end{remark}
\subsection{Radon-Nikod\'{y}m theorem.} 
Throughout this subsection, fix a stable measurable space $(X,\X)$, and let  $\mu$ be a stable probability measure on $(X,\X)$. 
If $f\geq 0$ is a stable integrable function with $\int_X f d\mu=1$, then from (D1)-(D3) and Theorem \ref{t:monotone convergence} we have that 
\[
 \nu(V|A):=\int_X 1_{V|A} f d\mu, \quad V|A\in \X, 
\]
is a stable probability measure absolutely continuous w.r.t.~$\mu$, where a stable probability measure $\nu$ is said to be \emph{absolutely continuous} w.r.t.~$\mu$ whenever $V|A\in \mathcal{X}$ and $\mu(V|A)= 0 $ imply $\nu (V|A)=0$. 
In this subsection, we prove the following converse statement. 
\begin{theorem}\label{t:rn}
If $\nu$ is a stable probability measure absolutely continuous relative to $\mu$, then there exists a stable integrable function $f\colon X\to L^0$ such that 
\[
 \nu(V|A)=\int_X 1_{V|A} f d\mu, 
\]
for all $V|A\in \X$
\end{theorem}
By a straightforward extension, a Radon-Nikod\'ym theorem can also be established in the case where $\mu$ and $\nu$ are $\sigma$-finite. 
The proof is based on the following auxiliary result. 
\begin{lemma}\label{l1}
Let $\mu_1$ and $\mu_2$ be finite stable measures on $(X,\mathcal{X})$ and define $\mu_3:=\mu_2-\mu_1$. 
Then there exists $ X_0|A_0\in \mathcal{X}$ such that 
\begin{itemize}
 \item[(i)] $\mu_3(X)\leq \mu_3( X_0|A_0 )$;
 \item[(ii)] $0\leq \mu_3(V|A)$ for all $V|A\in \mathcal{X}$ with $V|A\sqsubseteq  X_0|A_0$. 
\end{itemize}
\end{lemma}
\begin{proof}
First, we establish the weaker statement: For all $r\in L^0_{++}$ there is $ X_r|A_r\in \mathcal{X}$ such that 
\begin{itemize}
 \item[$(i)^\prime$] $\mu_3(X)\leq \mu_3( X_r|A_r)$;
 \item[$(ii)^\prime$] $-r<\mu_3(V|A)$ for all $V|A\in \mathcal{X}$ with $V|A\sqsubseteq  X_r|A_r$. 
\end{itemize} 
Let $B_0:=\{\mu_3(X)\leq 0\}$. 
If $B_0=\Omega$, then $ X_r|A_r:=\{\ast\}$ satisfies (i)$^\prime$ and (ii)$^\prime$. 
Otherwise $B_0^c\in \F_+$. 
In this case let 
\[
B_1:= \sup \{B^\prime\in \F\colon B^\prime\subseteq B^c_0, -r<\mu_3(V|A) \text{ for all } V|A\in\mathcal{X} \text{ with } V|A\sqsubseteq X|B^\prime\}. 
\]
By an exhaustion argument, $B_1$ is attained. 
If $B_1=B_0^c$, then $ X_r|A_r:=X|B_1$ fulfills the required. 
Indeed, by stability and (M1), 
\[
\mu_3(X)=\mu_3(X)|B_0 + \mu_3(X)|B_1\leq  0|B_0 + \mu_3(X)|B_1= \mu_3(X|B_1), 
\]
which shows (i)$^\prime$, and (ii)$^\prime$ follows from the definition of $B_1$. 
If $B_0^c\cap B_1^c\in \F_+$, then by the maximality of $B_1$ there exists $V_1|A_1\in\mathcal{X}$ with $V_1|A_1\sqsubseteq X|(B_0 \cup  B_1)^c$ such that $\mu_3(V_1|A_1)|(B_0 \cup  B_1)^c\leq -r|(B_0 \cup  B_1)^c$.  
Thus 
\begin{align*}
&\mu_3((V_1|A_1)^\sqsubset|(B_0 \cup  B_1)^c+X|B_1+\{\ast\}|B_0) &\\
&=\mu_3((V_1|A_1)^\sqsubset)|(B_0 \cup  B_1)^c+ \mu_3(X)|B_1 +0|B_0 \quad &\text{by stability and (M1)}\\
&=(\mu_3(X)-\mu_3(V_1|A_1))|(B_0 \cup  B_1)^c+\mu_3(X)|B_1 +0|B_0\quad &\text{by (M5)}\\
&\geq \mu_3(X)|(B_0 \cup  B_1)^c +\mu_3(X)|B_1 + 0|B_0 \quad & \text{by choice of $V_1|A_1$}\\
&\geq \mu_3(X). \quad & \text{by definition of $B_0$}. 
\end{align*}
So if 
\begin{equation*}
B_2:= \sup \{B^\prime\in \F\colon B^\prime\subseteq B^c_0\cap B_1^c, -r<\mu_3(V|A) \text{ for all } V|A\in\mathcal{X} \text{ with } V|A\sqsubseteq X|B^\prime\} 
\end{equation*}
is equal to $(B_0 \cup  B_1)^c$, then 
\[
X_r|A_r:=(V_1|A_1)^\sqsubset|(B_0 \cup  B_1)^c+X|B_1+\{\ast\}|B_0
\]
satisfies (i)$^\prime$ and (ii)$^\prime$. 
Otherwise we have $B_0^c \cap  B_1^c\cap B_2^c\in \F_+$ and there exists $V_2|A_2\in\mathcal{X}$ with $V_2|A_2\sqsubseteq (V_1|A_1)^\sqsubset|(B_0 \cup  B_1 \cup  B_2)^c$ such that $\mu_3(V_2|A_2)|(B_0 \cup  B_1 \cup  B_2)^c\leq -r|(B_0 \cup  B_1 \cup  B_2)^c$.  
Since $V_1|A_1\sqcap V_2|A_2= \{\ast\} $, one has similarly to the above computation  
\begin{equation*}
\mu_3((V_1|A_1\sqcup V_2|A_2)^\sqsubset|(B_0 \cup  B_1 \cup  B_2)^c+(V_1|A_1)^\sqsubset|B_2+X|B_1+\{\ast\}|B_0) \geq \mu_3(X), 
\end{equation*}
and we can repeat the previous procedure. 
If this procedure does not yield the desired after finitely many steps, we obtain a sequence $(B_n)_{n\geq 0}$ of pairwise disjoint elements in $\F$ and a sequence $(V_n|A_n)_{n\geq 1}$ of pairwise disjoint sets in $\mathcal{X}$ such that 
\begin{align*}
 &\mu_3((V_1|A_1\sqcup \ldots \sqcup V_n|A_n)^\sqsubset|(B_0 \cup  B_1 \cup  \ldots  \cup  B_n)^c \\ 
 &+(V_1|A_1\sqcup \ldots \sqcup V_{n-1}|A_{n-1})^\sqsubset|B_n+\ldots +(V_1|A_1)^\sqsubset|B_2+X|B_1+\{\ast\}|B_0) \geq \mu_3(X), 
\end{align*}
and 
\begin{equation*}
 \mu_3(V_n|A_n)|(B_0 \cup  B_1 \cup  \ldots  \cup  B_n)^c\leq -r|(B_0 \cup  B_1 \cup  \ldots  \cup  B_n)^c 
\end{equation*}
for all $n\geq 1$. 
If $B:=\cap_{n\geq 0} B_n^c\in \F_+$, then by (M2), 
\begin{align*}
 \mu_3(\sqcup_n V_n|A_n\cap B) =\sum_{n} \mu_3(V_n|A_n)|B+ 0|B^c = -\infty|B + 0|B^c 
\end{align*}
which contradicts the finiteness of $\mu_3$.  
Thus $(B_n)_{n\geq 0}$ is a partition of $\Omega$ and $ X_r|A_r:=\sum W_n|B_n$ where $W_0:= \{\ast\} $, $W_1:=X$ and $W_n:=(V_1|A_1\sqcup \ldots \sqcup V_{n-1}|A_{n-1})^\sqsubset$, $n\geq 2$, satisfies (i)$^\prime$ and (ii)$^\prime$.  

Second, we apply the previously established weaker statement to prove the claim. 
For every $n\in \mathbb{N}$, we can recursively choose $X_{1/n}|A_{1/n}$ such that $\mu_3(X)\leq \mu_3(X_{1/n}|A_{1/n})$ and $-1/n<\mu_3(V|A)$ for all $V|A\in\mathcal{X}$ with $V|A\sqsubseteq X_{1/n}|A_{1/n}$ and such that $X_{1/(n+1)}|A_{1/(n+1)}\sqsubseteq X_{1/n}|A_{1/n}$. 
Then $ X_0|A_0 :=\sqcap_n X_{1/n}|A_{1/n}$ fulfills (i)$^\prime$ and (ii)$^\prime$ due to (M8). 
\end{proof}
We prove Theorem \ref{t:rn}. 
\begin{proof}
Let $\mathcal{H}$ be the collection of all stable integrable functions $f\colon X\to  L^0_+$ such that 
\[
\int_X 1_{V|A} f d\mu\leq \nu (V|A)
\]
for all $V|A\in\mathcal{X}$. 
Inspection shows that $\mathcal{H}$ is upward directed. 
Let $r=\sup_{f\in\mathcal{H}}\int_X f d\mu \leq 1$. 
There exists an increasing sequence $(f_k)$ in $\mathcal{H}$ such that $\sup_k \int_X f_k d\mu=r$.
By Theorem \ref{t:monotone convergence}, we have $\int_X f d\mu=r$ for $f:=\sup_k f_k$. 
By another application of Theorem \ref{t:monotone convergence}, we see that 
\[
 \tilde{\nu}(V|A):=\int_X 1_{V|A} f d\mu \leq \nu (V|A) \quad \text{for all $V|A\in\mathcal{X}$}. 
\]
It remains to show that $\lambda:=\nu  - \tilde{\nu}=0$.  
By contradiction, suppose that $C:=\{\lambda(X)>0\}\in \F_+$. 
W.l.o.g. assume $C=\Omega$. 
Let $s:=\lambda(X)/2$.  
Applying Lemma \ref{l1} to $\mu_2=\lambda$ and $\mu_1=s\mu$ supplies us with $ X_0|A_0 \in\mathcal{X}$ such that $\lambda( X_0|A_0 )-s\mu( X_0|A_0 )\geq \lambda(X)-s\mu(X)>0$ and $\lambda(V|A)\geq s\mu(V|A)$ for all $V|A\in \mathcal{X}$ with $V|A\sqsubseteq  X_0|A_0$.  
It can be checked that $\tilde{f}:=f+s1_{ X_0|A_0 }\in \mathcal{H}$ with $\int_X \tilde{f} d\mu>r$ which is a contradiction. 
\end{proof}
\subsection{Daniell-Stone theorem and Riesz representation theorem.}
The Daniell-Stone theorem and the Riesz representation theorem provide sufficient conditions under which a positive linear functional on a vector lattice is an integral. 
In this last section, we establish a conditional version of both theorems which state when a positive $L^0$-linear function on a stable vector lattice is a stable integral. 
The proof of the Daniell-Stone theorem in the present setting is an adaptation of a proof of the classical statement by using Carath\'eodory's extension theorem, and therefore we omit its proof.  
We prove Riesz type representation and regularity results for positive $L^0$-linear functions on the stable vector lattice of all stable sequentially continuous functions $f\colon (L^0)^d\to L^0$ and its stable subspace consisting of functions with stably compact conditional support, respectively.  

Consider on $(L^0)^d=L^0(\mathbb{R}^d)$ the $L^0$-valued Euclidean norm, see Example \ref{exp:sets}.5), and endow it  with the stable Borel $\sigma$-algebra $\B^d:=\mathcal{B}((L^0)^d)$, see Example \ref{exp:algebra}.2). 
A conditional subset $V|A\sqsubseteq (L^0)^d$ is said to be 
\begin{itemize}
\item \emph{$L^0$-bounded} if there exists $M\in L^0_+$ such that $\|x|A\|=\|x\||A\leq M|A$ for all $x|A\in V|A$;
\item \emph{sequentially closed} if $V|A$ includes the limit of every a.s.~convergent sequence in $V|A$;
\item \emph{stably compact} if $V|A$ is $L^0$-bounded and sequentially closed.\footnote{Stable compactness refers to the interpretation of conditional compactness in a classical context. The characterization, which is stated above, is based on a conditional version of the Heine-Borel theorem, see \cite[Theorem 4.6]{DJKK13}.}
\end{itemize}
Let $X$ be a stable set of functions on $\Omega$.  
Recall that for a sequence $(f_k)$ of stable functions $f_k\colon X\to L^0$ and a partition $(A_k)$, we define $\sum f_k|A_k\colon X\to L^0$ by $(\sum f_k|A_k)(x):=\sum f_k(x)|A_k$, $x\in X$. We call a space $\mathcal{L}$ of functions $f\colon X\to L^0$ \emph{stable} if  $\mathcal{L}\neq \emptyset$ and $\sum f_k|A_k\in \mathcal{L}$ for all sequences $(f_k)$ in $\mathcal{L}$ and every partition $(A_k)$ of $\Omega$. 
\begin{definition}
Let $X$ be a stable set of functions on $\Omega$. 
A set $\mathcal{L}$ of stable functions $f\colon X\to L^0$ is said to be a \emph{stable Stone vector lattice} if $\mathcal{L}$ is stable and $f+rg,\min\{f,g\},\min\{f,1\}\in \mathcal{L}$ for all $f,g\in\mathcal{L}$ and $r\in L^0$.\footnote{Here, $1$ denotes the function with the constant value $1=1_\Omega\in L^0$.}
\end{definition}
\begin{examples}
Let $\text{cl}(V|A)$ denote the sequential closure of $V|A$. 
Let $f\colon (L^0)^d\to L^0$ be a stable function. Its \emph{conditional support} is defined by
\[
\text{supp}(f):=\text{cl}(f^{-1}(\{0\}^\sqsubset)), 
\]
and $f$ is said to have \emph{stably compact conditional support} if $\text{supp}(f)$ is stably compact. 
A function $f\colon (L^0)^d\to L^0$  is said to be \emph{sequentially continuous} if $f(x_n)\to f(x)$ a.s.~whenever $x_n\to x$ a.s. 
Let $\mathcal{C}$ denote the space of all stable and sequentially continuous functions, and $\mathcal{C}_c$ denote its subspace of functions with stably compact conditional support.  Then both $\mathcal{C}$ and $\mathcal{C}_c$ are stable Stone vector lattices. 
\end{examples}
Given a stable Stone vector lattice $\mathcal{L}$, a function $L\colon \mathcal{L}\to L^0$ is said to be 
\begin{itemize}
\item \emph{stable} if $L(\sum f_k|A_k)=\sum L(f_k)|A_k$ for all sequences $(f_k)$ in $\mathcal{L}$ and every partition $(A_k)$;
\item \emph{$L^0$-linear} if $L(f+rg)=L(f)+rL(g)$ for all $f,g\in \mathcal{L}$ and $r\in L^0$;   
\item \emph{continuous from above} if $L(f_n)\downarrow 0$ a.s.~whenever $f_n(x)\downarrow 0$ a.s.~for all $x\in X$.  
\end{itemize}
We have the following conditional version of the Daniell-Stone theorem. 
\begin{theorem}\label{t:ds}
Let $\mathcal{L}$ be a stable Stone vector lattice and $L:\mathcal{L}\to L^0$ stable, $L^0$-linear, continuous from above and such that $L(f)\geq 0$ whenever $f\geq 0$. 
Then there exists a stable measure $\mu$ on 
\[
\Sigma(\mathcal{L}):=\Sigma(\{f^{-1}(V|A)\colon V|A\in \B(L^0), f\in \mathcal{L}\})
\]
such that 
\[
L(f)=\int_X f d\mu \quad \text{for all }f\in \mathcal{L}. 
\]
\end{theorem}
In the following proofs of the conditional Riesz representation results it is necessary to pass from the sequential continuity of a sequence of stable functions to its sequential uniform continuity. 
This can be achieved by the following conditional version of Dini's theorem. 
In the statement of the next lemma, we use the following extension of a sequence of stable functions. 
If $(f_n)$ is a classical sequence of stable functions, we can extend $(f_n)$ to a stable net parametrized by $L^0_s(\N)$ by defining $f_n:=\sum f_{n_k}|A_k$ for $n=\sum n_k|A_k\in L^0_s(\N)$. 
\begin{lemma}\label{t:dini}
Let $X\subset (L^0)^d$ be stably compact and $(f_n)$ a decreasing sequence of stable and sequentially continuous functions $f_n\colon X\to L^0$. 
Let $f\colon X\to L^0$ be stable sequentially continuous such that $f_{n}(x)\to f(x)$ a.s.~for all $x\in X$. 
Then, for every $r\in L^0_{++}$ there exists $n_0\in L^0_s(\N)$ such that $\sup_{x\in X}\vert f_n(x)-f(x)\vert< r$ for all $n\geq n_0$. 
\end{lemma}
\begin{proof}
The statement can be proved similarly to the respective classical statement, see e.g.~\cite[Theorem 2.66]{aliprantis01}, by using the characterization of stable compactness in terms of stable open coverings, see \cite{DJKK13}. 
\end{proof}
We introduce the following regularity conditions for stable measures on $(L^0)^d$. 
\begin{definition}
A stable measure $\mu$ on $\mathcal{B}^d$ is called 
\begin{itemize}
\item \emph{conditionally closed regular} whenever 
  \begin{equation*}
\mu(V|A)=\sup\{\mu(W|B)\colon W|B\sqsubseteq V|A \text{ sequentially closed}\}
\end{equation*}
for all $V|A\in\mathcal{B}^d$; 
\item \emph{conditionally regular} whenever 
\begin{equation*}
\mu(V|A)=\sup\{\mu(W|B)\colon W|B\sqsubseteq V|A \text{ stably compact}\}
\end{equation*}
for all $V|A\in\mathcal{B}^d$.
\end{itemize}
\end{definition}
\begin{theorem}\label{t:riesz1}
Let $L:\mathcal{C}\to L^0$ be stable, $L^0$-linear and such that $L(f)\geq 0$ whenever $f\geq 0$. 
Then there exists a conditionally regular finite stable measure $\mu$ on $\mathcal{B}^d$ such that 
\[
L(f)=\int_{(L^0)^d} f d\mu \quad \text{for all }f\in \mathcal{C}. 
\]
\end{theorem}
\begin{proof}
Let $(f_l)$ be a sequence of stable functions in $\mathcal{C}$ such that $f_l\downarrow 0$. 
Let $V_k:=\{x\in  (L^0)^d \colon \|x\|\leq k\}$ for $k\in \N$, and notice that $V_k$ is stably compact. 
By triangle inequality, $d(\cdot,V_k):=\inf\{\|\cdot - y\|\colon y\in V_k\}$ is sequentially continuous, and by \cite[Theorem 4.4]{Cheridito2012}, the infimum is attained since it is also stable. 
Define $g_k:=\max\{1-d(\cdot,V_k),0\}$ and $h_{k,l}:=g_kf_l + (1-g_k)f_1\|\cdot\|/2k$ for all $k,l\in \N$. 
By definition, $f_l\leq h_{k,l}$ for all $k,l$. 
Fix $r\in L^0_{++}$. 
By changing, if necessary, to a stable net, let $k\in L^0_s(\N)$ be large enough such that $1/(2k) L(f_1(1-g_k)\| \cdot\|)<r/2$. 
By Lemma \ref{t:dini}, choose $l\in L^0_s(\N)$ sufficiently large such that $L(g_kf_l)<r/2$. 
We obtain 
\[
L(f_l)\leq L(g_kf_l) + 1/(2k) L((1-g_k)f_1\| \cdot\|)<r.
\]
By Theorem \ref{t:ds}, there exists a finite stable measure $\mu$ on $\mathcal{B}^d$ such that 
\[
L(f)=\int_{(L^0)^d} f d\mu \quad \text{for all }f\in \mathcal{C}.  
\]
As for the conditional tightness of $\mu$, let 
\begin{equation*}
\mathcal{G}:=\{V|A\in\mathcal{B}^d\colon V|A \text{ and } (V|A)^\sqsubset \text{ are conditionally closed regular}\}. 
 \end{equation*}
 It follows from (M5) that $(L^0)^d\in\mathcal{G}$, and thus from (S4) and the stability of $\mu$ that $\mathcal{G}$ is a stable collection. 
 By definition, $\mathcal{G}$ is closed under conditional complementation. 
 Let $(W_n|B_n)$ be a sequence in $\mathcal{G}$. 
 Fix $r\in L^0_{++}$.
 Let $Z_n|C _n\sqsubseteq W_n|B _n$ be sequentially closed with $\mu(W_n|B_n\sqcap (Z_n|C _n)^\sqsubset)<r/2^n$ for each $n$. 
 By (M6) and Boolean arithmetic, we obtain $\mu((\sqcup_n W_n|B_n)\sqcap(\sqcup_n (Z_n|C _n)^\sqsubset))<r$. 
 Since $r\in L^0_{++}$ is arbitrary, we have $\sqcup_n W_n|B_n\in\mathcal{G}$. 
 It follows from (M5) that also $(\sqcup_n W_n|B_n)^\sqsubset\in\mathcal{G}$. 
 It remains to show that every stable open ball is conditionally closed regular. 
 Let $B_r(x)$ be a stable open ball. 
 Define $U_n:=\{x\in (L^0)^d\colon d(x,B_r(x)^\sqsubset)\geq 1/n\}$. 
 Since $(U_n)$ is a sequence of stable sequentially closed sets with $\sqcup_n U_n=B_r(x)$, it follows from (M7) that $B_r(x)$ is conditionally closed regular. 
 Finally, for every $V|A\in \B^d$ we have $\mu(V|A\sqcap V_n)\uparrow \mu(V|A)$ due to (M7), where $V_n:=\{x\in  (L^0)^d \colon \|x\|\leq n\}$ for $n\in \N$. 
 Moreover, for every $n$ there exists a sequentially closed set $W_n|B_n\sqsubseteq V|A\sqcap V_n$ such that $\mu(W_n|B_n)\geq \mu(V|A\sqcap V_n)-1/n$. 
 Since a stable sequentially closed subset of a stably compact set is stably compact (see \cite[Proposition 3.27]{DJKK13}), it follows that $\mu$ is conditionally tight.  
\end{proof}
\begin{theorem}\label{t:riesz}
Let $L:\mathcal{C}_c\to L^0$ be stable, $L^0$-linear and such that $L(f)\geq 0$ whenever $f\geq 0$. 
Then there exists a stable measure $\mu$ on $\mathcal{B}^d$ such that 
\[
L(f)=\int_{ (L^0)^d } fd\mu \quad \text{ for all } f\in\mathcal{C}_c. 
\]
Moreover, it holds 
\begin{itemize}
\item $\mu([x,y])<\infty$ for all stably compact intervals $[x,y]:=\{z\in (L^0)^d\colon x \leq z \leq y\}$;\footnote{We consider on $(L^0)^d$ the order $(x_1,\ldots,x_d)\leq (y_1,\ldots,y_d)$ if and only if $x_i\leq y_i$ for all $i=1,\ldots,d$.}
\item $\mu(V|A)=\sup\{\mu(W|B)\colon W|B\sqsubseteq V|A \text{ stably compact}\}$ for all $V|A\in\mathcal{B}^d$ with $\mu(V|A)<\infty$.  
\end{itemize}
\end{theorem}
\begin{proof}
In order to obtain the assumptions of Theorem \ref{t:ds}, given a sequence $(f_n)$ of stable functions in $\mathcal{C}_c$ such that $f_n\downarrow 0$,  apply Lemma \ref{t:dini} to the  sequence $(1_{\text{supp}(f_1)}f_n)$. We have $\mu([x,y])\leq \int_{(L^0)^d}f d\mu= L(f)<\infty$, where $f=\max\{1-d(\cdot,[x,y]),0\}$. 
The regularity condition can be shown similarly to Theorem \ref{t:riesz1}. 
\end{proof}


\begin{thebibliography}{10}

\bibitem{aliprantis01}
C.~D. Aliprantis and K.~C. Border.
\newblock {\em Infinite Dimensional Analysis: a Hitchhiker's Guide}.
\newblock Springer, 2006.

\bibitem{BH14}
J.~Backhoff and U.~Horst.
\newblock {Conditional analysis and a Principal-Agent problem}.
\newblock {\em SIAM J. Financial Math.}, 7(1):477--507, 2016.

\bibitem{martinsamuel13}
T.~Bielecki, I.~Cialenco, S.~Drapeau, and M.~Karliczek.
\newblock {Dynamic assessement indices}.
\newblock {\em Stochastics}, 88(1):1--44, 2016.

\bibitem{kuppermaccheroni}
S.~Cerreia-Vioglio, M.~Kupper, F.~Maccheroni, M.~Marinacci, and N.~Vogelpoth.
\newblock {Conditional $L_p$-spaces and the duality of modules over
  $f$-algebras}.
\newblock {\em J. Math. Anal. Appl.}, 444(2):1045--1070, 2016.

\bibitem{ch11}
P.~Cheridito and Y.~Hu.
\newblock {Optimal consumption and investment in incomplete markets with
  general constraints}.
\newblock {\em Stoch. Dyn.}, 11(2):283--299, 2011.

\bibitem{Cheridito2012}
P.~Cheridito, M.~Kupper, and N.~Vogelpoth.
\newblock {Conditional analysis on $\mathbb{R}^d$}.
\newblock {\em Set Optimization and Applications, Proceedings in Mathematics \&
  Statistics}, 151:179--211, 2015.

\bibitem{cs12}
P.~Cheridito and M.~Stadje.
\newblock {BS$\Delta$Es and BSDEs with non-Lipschitz drivers: comparison,
  convergence and robustness}.
\newblock {\em Bernoulli}, 19(3):1047--1085, 2013.

\bibitem{DJ13}
S.~Drapeau and A.~Jamneshan.
\newblock {Conditional preferences and their numerical representations}.
\newblock {\em J.~Math.~Econom.}, 63:106--118, 2016.

\bibitem{DJKK13}
S.~Drapeau, A.~Jamneshan, M.~Karliczek, and M.~Kupper.
\newblock {The algebra of conditional sets, and the concepts of conditional
  topology and compactness}.
\newblock {\em J.~Math.~Anal.~Appl.}, 437(1):561-- 589, 2016.

\bibitem{drapeau2017fenchel}
S.~Drapeau, A.~Jamneshan, and M.~Kupper.
\newblock {A Fenchel-Moreau theorem for $\bar L^0$-valued functions}.
\newblock {\em Preprint available at arXiv:1708.03127}, 2017.

\bibitem{DKKM13}
S.~Drapeau, M.~Karliczek, M.~Kupper, and M.~Streckfuss.
\newblock {Brouwer fixed point theorem in $(L^0)^d$}.
\newblock {\em J. Fixed Point Theory Appl.}, 301(1), 2013.

\bibitem{kupper03}
D.~Filipovi{\'c}, M.~Kupper, and N.~Vogelpoth.
\newblock {Separation and duality in locally $L^0$-convex modules}.
\newblock {\em J.~Funct.~Anal.}, 256:3996--4029, 2009.

\bibitem{kupper11}
D.~Filipovi{\'c}, M.~Kupper, and N.~Vogelpoth.
\newblock {Approaches to conditional risk}.
\newblock {\em SIAM J. Financial Math.}, 3(1):402--432, 2012.

\bibitem{frittelli14}
M.~Frittelli and M.~Maggis.
\newblock {Complete duality for quasiconvex dynamic risk measures on modules of
  the $L^{p}$-type}.
\newblock {\em Stat.~Risk Model.}, 31(1):103--128, 2014.

\bibitem{frittelli2014conditional}
M.~Frittelli and M.~Maggis.
\newblock Conditionally evenly convex sets and evenly quasi-convex maps.
\newblock {\em J. Math. Anal. Appl.}, 413:169--184, 2014.

\bibitem{guo10}
T.~Guo.
\newblock {Relations between some basic results derived from two kinds of
  topologies for a random locally convex module}.
\newblock {\em J.~Funct.~Anal.}, 258:3024--3047, 2010.

\bibitem{HR1987conditioning}
L.~P. Hansen and S.~F. Richard.
\newblock {The Role of Conditioning Information in Deducing Testable
  Restrictions Implied by Dynamic Asset Pricing Models}.
\newblock {\em Econometrica}, 55:587--613, 1987.

\bibitem{jamneshan2014topos}
A.~Jamneshan.
\newblock {Sheaves and conditional sets}.
\newblock {\em Preprint available at arXiv:1403.5405}, 2014.

\bibitem{JKZ2017control}
A.~Jamneshan, M.~Kupper, and J.~M. Zapata.
\newblock {Parameter-dependent stochastic optimal control in finite discrete
  time}.
\newblock {\em Preprint available at arXiv: 1705.02374}, 2017.

\bibitem{jamneshan2017compact}
A.~Jamneshan and J.~M. Zapata.
\newblock {On compactness in $L^0$-modules}.
\newblock {\em Preprint available at arXiv:1711.09785}, 2017.

\bibitem{Kallenberg2002}
O.~Kallenberg.
\newblock {\em {Foundations of Modern Probability}}.
\newblock Probability and its Applications (New York). Springer-Verlag, New
  York, second edition, 2002.

\bibitem{monk1989handbook}
J.~Monk, S.~Koppelberg, and R.~Bonnet.
\newblock {\em Handbook of Boolean Algebras}.
\newblock Handbook of Boolean Algebras: General Theory of Boolean Algebras.
  North-Holland, 1989.

\bibitem{OZ2017stabil}
J.~Orihuela and J.~M. Zapata.
\newblock {Stability in locally $L^0$-convex modules and a conditional version
  of James' compactness theorem}.
\newblock {\em J. Math. Anal. Appl.}, 452(2):1101--1127, 2017.

\bibitem{takeuti2015two}
G.~Takeuti.
\newblock {\em Two Applications of Logic to Mathematics}.
\newblock Publications of the Mathematical Society of Japan. Princeton
  University Press, 2015.

\end{thebibliography}
\end{document}